\documentclass[11pt]{article}

\usepackage{amsmath,amsfonts,latexsym, xspace,amsthm,graphicx, amssymb}
\usepackage{enumerate}
\usepackage[letterpaper,left=1in,right=1in,top=1in,bottom=1in]{geometry}

\usepackage{url}

\title{Modularity of regular and treelike graphs}

\author{Colin McDiarmid and Fiona Skerman}

\newcommand{\Gr}{G_{n,r}}

\DeclareMathAlphabet{\mathcal}{OMS}{cmsy}{m}{n}

\usepackage{amsmath}
\usepackage{amssymb}
\usepackage{amsthm}

\usepackage{tikz}

\usetikzlibrary{arrows}
\newcommand{\midarrow}{\tikz \draw[-triangle 90] (0,0) -- +(.1,0);}

\newcommand{\cA}{\mathcal A}
\newcommand{\cC}{\mathcal C}
\newcommand{\cG}{\mathcal G}

\newcommand{\eps}{\varepsilon}
\newcommand{\ds}{{\rm vol}}
\renewcommand{\deg}{{\rm deg}}
\newcommand{\tw}{{\rm tw}}
\newcommand{\bw}{{\rm bw}}

\newcommand{\q}{q^*}

\newtheorem{thm}{Theorem}
\newtheorem{theorem}[thm]{\sc Theorem}

\newtheorem{proposition}[thm]{\sc Proposition}

\newtheorem{cor}[thm]{\sc Corollary}
\newtheorem{lemma}[thm]{\sc Lemma}
\newtheorem{conj}[thm]{\sc Conjecture}
\newtheorem{fact}[thm]{\sc Fact}

\begin{document}
\maketitle

\begin{abstract}
{Clustering algorithms for large networks typically use modularity values to test which partitions of the vertex set better represent structure in the data.  The \emph{modularity} of a graph is the maximum modularity of a partition. We consider the modularity of two kinds of graphs.

For $r$-regular graphs with a given number of vertices, we investigate the minimum possible modularity, the typical modularity, and the maximum possible modularity.  In particular, we see that for random cubic graphs the modularity is usually in the interval $(0.666, 0.804)$, and for random $r$-regular graphs with large $r$ it usually is of order $1/\sqrt{r}$.
These results help to establish baselines for statistical tests on regular graphs.}
 
The modularity of cycles and low degree trees is known to be close to 1: we extend these results to `treelike' graphs, where the product of treewidth and maximum degree is much less than the number of edges.  This yields for example the (deterministic) lower bound $0.666$ mentioned above on the modularity of random cubic graphs.
\end{abstract}

\section{Introduction and Statement of Results} 
The recently greater availability of data on large networks in many fields has led to increasing interest in techniques to discover network structure. In the analysis of these networks, clusters or communities found using modularity optimisation have become a focus of study. Thus we need benchmarks to assess the statistical significance of observed community structure~\cite{trulymodular}.

Further, the popularity of modularity-based clustering techniques~\cite{fortunato2016community,popular} and the link to the Potts model in statistical physics~\cite{reichardt2006statistical} have prompted much research into the modularity of graphs from various classes. The asymptotic value of the modularity of each of the following graph classes has been shown to approach the maximum value 1; cycles~\cite{nphard}, low degree trees~\cite{bagrow,modgraphclasses} and lattices~\cite{GPA04}.

In this paper we focus on the (maximum) modularity $q^*(G)$ of a graph $G$ (precise definitions are given later) from one of two natural related and contrasting areas, namely regular graphs and treelike graphs.

We think of a graph as \emph{treelike} if by deleting a few edges we may obtain a graph with low treewidth (treewidth measures how much we have  to `fatten' a tree to contain the graph). We show that if a graph $G$ with many edges has low maximum degree, and by deleting a small proportion of its edges we may obtain a graph with low treewidth, then $G$ has high modularity.  This result much extends the results mentioned earlier about cycles and trees; it shows that random planar graphs have modularity asymptotically 1; and it shows that every cubic (3-regular) graph has modularity at least about $2/3$.

For $r$-regular graphs with a given number $n$ of vertices, we investigate the minimum possible modularity, the typical modularity, and the maximum possible modularity.  For example, consider a random cubic graph $G_{n,3}$.  Locally, looking out a fixed distance from a random vertex, it is a tree with high probability (whp), though globally it is far from treelike.  We shall see that $q^*(G_{n,3}) \leq 0.804$ whp,
and simulations suggest that the value may typically not be much above the deterministic lower bound of about~$2/3$. 
In fact, we consider random $r$-regular graphs $G_{n,r}$ for each $r$ from 3 to 12 (see Table~\ref{tab:rreg}); and also show that when $r$ is large $q^*(\Gr)$ is contained whp in an interval that scales with $1/\sqrt{r}$.


\subsection{Modularity of a graph}
The definition of modularity was first introduced by Newman and Girvan in~\cite{NewmanGirvan}. Many or indeed most popular algorithms used to search for clusterings on large datasets are based on finding partitions with high modularity~\cite{fortunato2016community,popular}. See~\cite{fortunato2010community,porter2009communities} for surveys on community detection including modularity based methods.

In order to define the modularity of a graph $G$, we first define the modularity $q_\cA(G)$ for a partition $\cA$ of its vertex set.  This is a measure designed to score highly when most edges fall within the parts but to be penalised when some parts have large sums of degrees.
Denote the number of edges in the subgraph induced by vertex set~$A$ by $e(A)$, and let the \emph{volume $\ds (A)$ of $A$ be the sum of the degrees (in the whole graph $G$) of the vertices in $A$.} 

Let $G$ be a graph with $m\geq 1$ edges. For a vertex partition $\cA$ of~$G$, we define
\[ \begin{split}
q_\cA(G) & = 
\frac{1}{2m}\sum_{A \in \cA} \sum_{u, v \in A} \left( {\mathbf 1}_{uv\in E} - \frac{\deg(u) \deg(v)}{2m} \right)\\
&=
\frac{1}{m}\sum_{A \in \cA} e(A) - \frac{1}{4m^2}\sum_{A \in \cA} \ds (A)^2 \; = \;
q_\cA^E(G) - q_\cA^D(G),
\end{split} \]
where the \emph{edge-contribution} (or \emph{coverage}) $q_\cA^E(G)$ and the \emph{degree-tax}  $q_\cA^D(G)$ are given by
\[ q_\cA^E(G) = \frac{1}{m}\sum_{A \in \cA} e(A) \;\;\; \mbox{ and } \;\;\; q_\cA^D(G) =    \frac{1}{4m^2}\sum_{A \in \cA} \ds (A)^2. \]
The \emph{modularity} $\q(G)$ of the graph $G$ is defined by $\q(G)=\max_{\cA}(G)$, where the maximum is over all partitions $\cA$ of the vertex set.
\smallskip

By definition we have $0 \leq \q(G) <1$ for each non-trivial graph $G$. For example, complete graphs, stars and more generally all complete multipartite graphs have modularity~0 (as noted in \cite{nphard}, \cite{modgraphclasses} and~\cite{bolla2013largest,majstorovic2014note} respectively); and it was shown recently~\cite{ERus,trajanovski2012maximum} that near-complete graphs also have modularity 0.  At the other extreme, the $n$-cycle $C_n$ has modularity near 1, see~(\ref{eqn.Cn}) below. A graph $G$ with no edges is defined to have modularity $1$ for any partition $\cA$~\cite{nphard}.


\subsection{Modularity of regular graphs}
\label{subsec.regmod}

Let $\cG(n,r)$ denote the set of all $r$-regular graphs with $n$ vertices, say with vertex set $\{1,\ldots,n\}$.  
It is easy to see that this set is non-empty if and only if $n \geq r+1$ and $rn$ is even.  Assume that this condition holds (as we shall often do implicitly).
Define $q_r^-(n)$ to be the minimum modularity $\q(G)$ for an $n$-vertex $r$-regular graph $G$, that is
\[ q_r^-(n) = \min \{ \q(G) : G \in \cG(n,r) \} \]
and similarly let 
\[ q_r^+(n) = \max \{ \q(G) : G \in \cG(n,r) \}. \]
Also let $\Gr$ 
 denote a random graph sampled uniformly from $\cG(n,r)$.  We shall be particularly interested in the behaviour of the random variable $\q(\Gr)$.
When an event holds with probability tending to 1 as $n\rightarrow \infty$ we say that it holds \emph{with high probability} (\emph{whp}). 
Our main results on regular graphs are Theorems~\ref{thm.rlarge-min},~\ref{thm.kwus} and~\ref{thm.rlarge-random}.
Let us consider first $q_r^-(n)$ and $\q(\Gr)$, which are closely related, and then $q_r^+(n)$.

To introduce the results, we start with the simpler cases $r=1$ and $r=2$.
The case $r=1$ is trivial, as an $n$-vertex 1-regular graph $G$ must consist of $n/2$ disjoint edges. Since each part in an optimal partition must induce a connected graph, it is easy to check that the unique optimal partition has one part for each edge, and $q^*(G)= 1-2/n$.
The case $r=2$ is not trivial, though it is simpler than for $r \geq 3$: we give it a separate subsection.


\subsubsection{2-regular graphs} 
\label{subsec.r2}
 
For the minimum modularity $q_2^-(n)$, note first that $\tfrac{5}{\sqrt{6}} \approx 2.041$. 
\begin{proposition} \label{prop.minq}
The minimum modularity $q_2^-(n)$ satisfies 
$q_2^-(n) = 1- \tfrac{5}{\sqrt{6n}} +O(\tfrac1{n})$. 
\end{proposition} 


\begin{proposition}\label{prop.2reg}
The random 2-regular graph $G_{n,2}$ satisfies
\[ \q(G_{n,2}) = 1- \tfrac{2}{\sqrt{n}} + o(\tfrac{\log^2n}{n}) \;\; \mbox{ whp}.\]
\end{proposition}
For the $n$-cycle $C_n$, by~\cite{nphard} [Theorem~6.7] (see also the comments following Proposition~\ref{prop.connub} below), we have 
\begin{equation} \label{eqn.Cn}
\q(C_n) = 1- \tfrac{2}{\sqrt{n}} + O(\frac1{n}). 
\end{equation}
Thus whp $\q(G_{n,2})$ is extremely close to $\q(C_n)$. Propositions~\ref{prop.minq} and~\ref{prop.2reg} show that whp $\q(G_{n,2})$ 
is very close to the minimum possible value $q_2^-(n)$, indeed whp
\[ \q(G_2) - q_2^-(n) \sim  \tfrac{2}{\sqrt{n}} - \tfrac{5}{\sqrt{6n}} \approx \tfrac{0.04}{\sqrt{n}}.\]

It is easier to determine the maximum modularity $q_2^+(n)$.  For example,
$q_2^+(n) = 1-\tfrac{3}{n}$
when $n$ is divisible by 3, and is attained by $n/3$ disjoint triangles.
Indeed we have a full story: 
\begin{proposition}\label{prop.2regplus}
For each $n \geq 3$
\begin{equation}\label{eq.maxq2} 
q_2^+(n) = 
\begin{cases}
1-\tfrac{3}{n} & \mbox{ if } n \equiv 0 \mod{3}\\
1-\tfrac{3}{n} - \frac{4}{n^2} & \mbox{ if } n \equiv 1 \mod{3}\\
1-\tfrac{3}{n} - \frac{8}{n^2} & \mbox{ if } n \equiv 2 \mod{3}
\end{cases}
\end{equation}
For $n \geq 3$, the maximum value is attained if and only if:
$G$ is a disjoint union of $n/3$ copies of $C_3$ when $n \equiv 0 \mod{3}$;
$G$ is a disjoint union of one $C_4$ and $(n-4)/3$ $C_3$'s when $n \equiv 1 \mod{3}$; and
$G$ is a disjoint union of two $C_4$'s and $(n-4)/3$ $C_3$'s when $n \equiv 2 \mod{3}$ and $n \geq 8$, and $G$ is a five cycle $C_5$ when $n=5$.
\end{proposition}


\subsubsection{Minimum possible modularity $q_r^-(n)$}

For $r=3,\ldots,8$ the following deterministic result gives the best lower bound we know to hold whp for $\q(\Gr)$. (These lower bounds were originally proved in~\cite{McDSker} to hold whp using a Hamilton cycle construction.) It will follow quickly from Theorem~\ref{thm.tw} on treewidth and maximum degree.
\begin{proposition}\label{cor.allrreg}
For each $r\geq 2$,
\[ q_r^-(n) \geq \frac2{r} - 2 \sqrt{\frac6{n}} \]
for each possible value of $n$.
\end{proposition}

Now consider large $r$.
\begin{theorem}\label{thm.rlarge-min}
There is a constant $c>0$ such that, for each positive integer $r$,
\[ q_r^-(n) \geq c/\sqrt{r} \]
 for each sufficiently large $n$ (with $rn$ even).
\end{theorem}
For each $r \geq 3$, our best upper bounds on $q_r^-(n)$ will come from $\q(G_{n,r})$.


\subsubsection{Random modularity $\q(G(n,r))$}

How should we assess the statistical significance of clusters observed in regular networks? There has been recent interest in estimating the modularity of random graphs~\cite{frankearxiv,ERus,trulymodular,trajanovski2012maximum}. In order to tell if a given partition shows statistically significant clustering in a network it is natural to compare the modularity score to that of a corresponding random graph model~\cite{fortunato2016community,trulymodular}. We give results which bound the modularity of random $r$-regular graphs.  In 
Theorem~\ref{thm.kwus}, we consider small values of $r$; and improve results in~\cite{McDSker}.  After that, in Theorem~\ref{thm.rlarge-random} we consider larger values of $r$.

\begin{theorem}\label{thm.kwus}
For $r=3,\ldots, 12$, the modularity of a random $r$-regular graph $\Gr$ whp lies in the range indicated in Table~\ref{tab:rreg} (see also Figure 1). In particular, $0.666< q^*(G_{n,3}) < 0.804$ whp.  
\end{theorem}
 The lower bounds given for $r=3$ to 8 in Table~\ref{tab:rreg} (and Figure~\ref{fig.qsim}) are deterministic bounds from Proposition~\ref{cor.allrreg}.

\begin{theorem}\label{thm.rlarge-random}
For each (fixed) integer $r \geq 3$, the random $r$-regular graph $\Gr$ satisfies
\[ \q(\Gr) < 2/\sqrt{r} \; \mbox{ whp}; \]
and there is a constant $r_0$ such that for each $r \geq r_0$
\[ \q(\Gr) > 0.7631/\sqrt{r} \; \mbox{ whp}. \]
\end{theorem}

\begin{table*}
\begin{centering}
 $
\begin{array}{lrlllllllllll}
&\;\;r=			&& 3& 4 & 5 & 6 & 7 & 8& 9& 10 & 11 & 12\\
\hline
\mbox{\small bounds}&\q(G_{n,r}) > 		
&& 0.666 & 0.499 & 0.399 & 0.333 & 0.285 & 0.249	 & 0.226 & 0.214	& 0.204 & 0.196 \\
&\q(\Gr) < 	
&& 0.804 & 0.684 & 0.603 & 0.544 & 0.499 & 0.463 &0.433 & 0.408 & 0.388 & 0.370 \\
\\
\mbox{\small simulations}\!\!& \mbox{\small Louvain} 	&& 0.679 & 0.531 & 0.440& 0.380 &0.343 &0.312&0.284\! &0.262&0.244&0.230\\
& \mbox{\small Reshuffle}  	&& 0.677 & 0.531 & 0.446& 0.391 &0.353 &0.326 &0.303 &0.285 &0.269 &0.256 \\
\end{array}
$
{\caption{Upper rows: lower and upper whp bounds on the modularity of random regular graphs $\Gr$, for Theorem~\ref{thm.kwus}. Lower rows: average modularity found in simulations on randomly generated $r$-regular graphs with $10000$ nodes, using the Louvain method~\cite{louvain} and a method `Reshuffle' from~\cite{martelot}, see Section~\ref{sec.sim}.
	 }\label{tab:rreg}}
\end{centering}
\end{table*} 

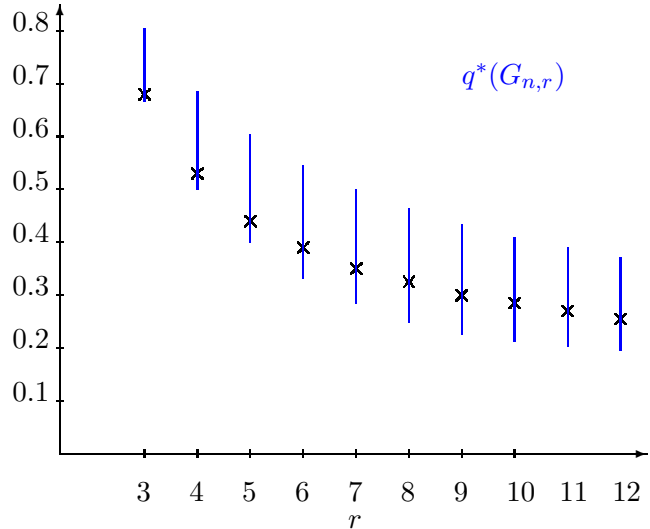
\begin{figure}
\begin{centering}

\begin{picture}(240,210)(20,-30)
\put(28,0){\vector(0,1){170}}
\put(28,0){\vector(1,0){222}}
\multiput(26.5,20)(0,20){8}{{\line(1,0){3}}}
\multiput(60,-1.5)(20,0){8}{{\line(0,1){3}}}
\put(10, 20){$0.1$}\put(10, 40){$0.2$}\put(10, 60){$0.3$}\put(10, 80){$0.4$}\put(10, 100){$0.5$}\put(10, 120){$0.6$}\put(10, 140){$0.7$}\put(10,160){$0.8$}  
\put(57, -18){$3$}\put(77, -18){$4$}\put(97, -18){$5$}\put(117, -18){$6$}\put(137, -18){$7$}\put(157, -18){$8$}\put(177, -18){$9$}\put(197, -18){$10$}\put(217, -18){$11$}\put(237, -18){$12$}

\color{blue}
\put(180, 140){$\q(\Gr)$}
\color{black}

\put(137, -28){$r$}


\put(60,136){\qbezier(-2,-2)(1,1)(2,2)\qbezier(2,-2)(1,-1)(-2,2)}
\put(80,106){\qbezier(-2,-2)(1,1)(2,2)\qbezier(2,-2)(1,-1)(-2,2)}
\put(100,88){\qbezier(-2,-2)(1,1)(2,2)\qbezier(2,-2)(1,-1)(-2,2)}
\put(120,78){\qbezier(-2,-2)(1,1)(2,2)\qbezier(2,-2)(1,-1)(-2,2)}
\put(140,70){\qbezier(-2,-2)(1,1)(2,2)\qbezier(2,-2)(1,-1)(-2,2)}
\put(160,65){\qbezier(-2,-2)(1,1)(2,2)\qbezier(2,-2)(1,-1)(-2,2)}
\put(180,60){\qbezier(-2,-2)(1,1)(2,2)\qbezier(2,-2)(1,-1)(-2,2)}
\put(200,57){\qbezier(-2,-2)(1,1)(2,2)\qbezier(2,-2)(1,-1)(-2,2)}
\put(220,54){\qbezier(-2,-2)(1,1)(2,2)\qbezier(2,-2)(1,-1)(-2,2)}
\put(240,51){\qbezier(-2,-2)(1,1)(2,2)\qbezier(2,-2)(1,-1)(-2,2)}

\color{blue}

\qbezier(60, 133.4)(60, 140)(60, 160.8)  
\qbezier(80, 100)(80, 120)(80, 136.8)     
\qbezier(100, 80)(100, 100)(100, 120.6) 
\qbezier(120, 66.6)(120, 100)(120, 108.8)    	
\qbezier(140, 57)(140, 90)(140, 99.8)      
\qbezier(160, 50)(160, 55)(160, 92.6)        
\qbezier(180, 45.2)(180, 50)(180, 86.6)      
\qbezier(200, 42.8)(200, 50)(200, 81.6)     
\qbezier(220, 40.8)(220, 50)(220, 77.6)     
\qbezier(240, 39.2)(240, 50)(240, 74)     

\end{picture}

\end{centering}

\caption{Simulation results for $n=10000$ nodes and proven theoretical bounds for random $r$-regular graphs with degrees $r=3,\ldots,10$. Each cross indicates the better computed modularity returned by two methods (see Section~\ref{sec.sim}), averaged over ten sampled graphs with $10000$ nodes. 
Theorem~\ref{thm.kwus} says that the modularity of a random regular graph $\Gr$ whp lies in the interval shown.}\label{fig.qsim}
\end{figure}

For large $r$ the theorem shows that the modularity of a random $r$-regular graph $\q(\Gr)$ whp tracks the minimum modularity $q^-_r(n)$, in the sense that both are of order $\Theta(1/\sqrt{r})$. 
For each given $r\geq 3$, it is an interesting open question how close $q^*(\Gr)$ typically is to $q_r^-(n)$. Can we construct $r$-regular graphs with modularity less than that for random graphs $\Gr$?

\subsubsection{Maximum possible modularity $q_r^+(n)$}

The maximum possible modularity $q_r^+(n)$ is easier to handle than $q_r^-(n)$ and $\q(\Gr)$, though it may be less important than them.  We have already discussed the much easier case $r=2$, and in Proposition~\ref{prop.2regplus} we gave a full story for this case. In Proposition~\ref{prop.maxqr} below, we see that in general we have $q^+_r(n) \leq 1-\frac{r+1}{n}$, and we define a function $g_r(n)$ such that $g_r(n) = 1 - \frac{r+1}{n} - O(1/n^2)$, and $q^+_r(n) = g_r(n)$ for sufficiently large~$n$.  Also, we identify the most modular $r$-regular graphs on $n$ vertices (apart from some small values). Proposition~\ref{prop.maxqr} contains and extends Proposition~\ref{prop.2regplus} except for the `$n$ sufficiently large' qualification.

These results complement the extremal results in~\cite{FortBart2008} and~\cite{trajanovski2012maximum}. The most modular connected graphs with a given number of edges were discussed in~\cite{FortBart2008}; and later~\cite{trajanovski2012maximum} investigated the most `$k$-modular' connected graphs parameterised by the number $m$ of edges and the number $k$ of parts, as well as the most 2-modular graphs parameterised by $m$ and the number $n$ of vertices.

\begin{proposition}\label{prop.maxqr}
Let $r \geq 2$ be an integer.

(a) The maximum modularity $q^+_r(n)$ of any $r$-regular graph on $n$ vertices satisfies
$$q^+_r(n) \leq 1-\frac{r+1}{n}$$
with equality achieved if and only if $r+1$ divides $n$; and for this case the unique way to attain the optimum is for the graph to consist of $n/(r+1)$ disjoint copies of $K_{r+1}$, with the connected components partition~$\cC$.  

(b) Given a positive integer $n$, write $n$ as $a(r+1)+b$ where $a$ is a non-negative integer and  $0 \leq b \leq r$, and let
\begin{equation} \label{eqn.qplus}
g_r(n) = 1-\frac{r+1}{n} - \frac{b\,(r+2+{\mathbf 1}_{r \mbox{ \small  odd}})}{n^2}.
\end{equation}
If $rn$ is even and $n$ is sufficiently large (for example, if $n \geq 9 r (r+1)^2$) then $q^+_r(n) = g_r(n)$, and this value is attained exactly for the graphs described in the proof, with the connected components partition~$\cC$.
\end{proposition}
It will follow easily from the above result and its proof -- see the comment following (\ref{eqn.grngeq}) -- that for each $n\geq r(r+2)$ we have
\begin{equation}\label{eq.approxbestmod}
q^+_r(n)\geq 1- \frac{r+2}{n},
\end{equation}
so the upper bound (from part (a)) and the lower bound differ by at most $1/n$.
Perhaps the upper bound 
holds for a much wider range of graphs?  It seems likely that it holds for graphs with minimum degree at least $r$.  We go further, and make a plausible but more speculative conjecture.

\begin{conj}For each integer $r \geq 1$, if an $n$-vertex graph has average degree at least $r$, then $\q(G) \leq 1 - \frac{r+1}{n}$.
\end{conj}

\smallskip

The maximum modularity is considerably smaller if we consider only connected graphs.
\begin{proposition} \label{prop.connub}
Let the graph $G$ have $m \geq 1$ edges.  If $G$ is connected  then
\[ \q(G) \leq 1  - \tfrac{2}{\sqrt{m}} +\tfrac1{m} ,\]
and if $G$ is 2-edge-connected  then
\[ \q(G) \leq 1  - \tfrac{2}{\sqrt{m}}. \]
\end{proposition}

The last result can be tight.  Since $C_n$ is 2-connected, it gives $\q(C_n) \leq 1- 2/\sqrt{n}$.  If $n=t^2$ for an integer $t \geq 2$, then partitioning into $t$ paths of $t$ vertices shows that $\q(C_n) = 1- 2/\sqrt{n}$.
In general, partitioning into $t= \lceil \sqrt{n} \rceil$ paths each with $t-1$ or $t$ vertices yields equation~(\ref{eqn.Cn}).


\subsection{Treewidth and maximum degree}
Bagrow makes a study of the modularity of some trees and treelike graphs in~\cite{bagrow}. He shows that Galton-Watson trees and $k$-ary trees have modularity tending to one. In~\cite{modgraphclasses} it is shown that any tree with maximum degree $\Delta(G)=o(n^{1/5})$ has asymptotic modularity one.  We shall see that this result extends to all trees with $\Delta(G)=o(n)$; and indeed it extends to all low degree graphs which are `treelike', in that they are `close' to graphs with low treewidth. This forms Theorem~\ref{thm.tw}, our main result in this section.

Treewidth is a central notion in the study of graphs and the design of algorithms~\cite{kleinberg2005textbook}: see~\cite{karb} for a survey. Let us recall the definitions. A \emph{tree-decomposition} of a graph $G=(V,E)$ is a pair consisting of a tree $T=(I, F)$ and a family $( X_i :  i \in I ) $ of subsets of $V$ (`bags'), one for each node $i$ of $T$, such that \vspace{-1mm}
\begin{enumerate}
	\item{} $\cup_{i \in I} X_i = V$
	\item{} for each edge $vw \in E$ there is a node $i \in I$ such that $v, w\in X_i$
	\item{} for all nodes $i,j,k\in I$, if $j$ is on the path between $i$ and $k$ in $T$, then $X_i \cap X_k \subseteq X_j$.
\end{enumerate}

The \emph{width} of a tree decomposition is $\max_{i \in I}  |X_i| - 1$; and the \emph{treewidth} $\tw(G)$ of a graph $G$ is the minimum width over all tree decompositions of $G$. Thus trees have treewidth 1, and indeed they are exactly the connected graphs with treewidth 1. Cycles have treewidth 2; and the graphs with treewidth at most 2 are exactly the series-parallel graphs.

The following result is our key tool for lower bounding the modularity of graphs which have small degrees, and which have small treewidth or can can be made so by deleting a few edges.


\begin{theorem}\label{thm.tw}
	Let $G$ be a graph with $m\geq 1$ edges and maximum degree $\Delta=\Delta(G)$, and let $E'$ be a subset of the edges such that $\tw(G\backslash E')\leq t$.	Then the modularity $\q(G)$ satisfies
	$$ \q(G) \geq 1- 2((t+1)\Delta/m)^{1/2}-|E'|/m.$$
\end{theorem}

Proposition~\ref{cor.allrreg} on $q_r^-(n)$ is a corollary which we shall deduce quickly from Theorem~\ref{thm.tw}.
Our second corollary of Theorem~\ref{thm.tw} is immediate.
\begin{cor}\label{cor.tw}
	For $m=1,2, \ldots$ let $G_m$ be a graph with $m$ edges.
	If $\tw(G_m) \cdot \Delta(G_m) = o(m)$ then
	$\q(G_m) \to 1$ as $m \to \infty$.
\end{cor}

This result is best possible, in that we cannot replace $o(m)$ by $O(m)$: here are two examples.

(a) If $G$ is the star $K_{1,m}$ (with treewidth 1 and maximum degree m) then
$\tw(G) \cdot \Delta(G) = 1 \cdot m = m$ and $\q(G)=0$~\cite{modgraphclasses}.

(b) For the random cubic graph $G = G_{n,3}$ on $n$ vertices (with $m=3n/2$) we have $\tw(G) \cdot \Delta(G) = 3 \, \tw(G) = O(m)$. However, by Theorem~\ref{thm.kwus}, $\q(G_{n,3}) \leq 0.804$ whp. 

Corollary~\ref{cor.tw} shows that a random planar graph $G_n$ with $n$ vertices whp has modularity near 1. For $\tw(G_n)=O(\sqrt{n})$ by~\cite{norin, gilbert1984separator}, whp $\Delta(G_n)=O(\log n)$~\cite{maxDegSurface}, and whp $m=\Theta(n)$; and so whp $\q(G_n) \geq 1 - O((\log n)^{\frac12}/n^{\frac14}) \; = 1-o(1).$  The same also holds for random graphs on any fixed surface.
\medskip

\noindent
\emph{Plan of the paper}

In the next section we briefly discuss our simulations. In the following section, we prove Propositions~\ref{prop.minq},~\ref{prop.2reg} and~\ref{prop.2regplus} which concern 2-regular graphs. After that, in Section~\ref{sec.tw-proofs}, we prove Theorem~\ref{thm.tw} on treelike graphs and show that this implies Proposition~\ref{cor.allrreg}. Theorem~\ref{thm.rlarge-min} on minimum modularity $q_r^-(n)$, and Theorems~\ref{thm.kwus} and~\ref{thm.rlarge-random} on random modularity $\q(\Gr)$ are all proven in Section~\ref{sec.minrand-proofs}. Then in Section~\ref{sec.maxmod-proofs} we prove our results on maximum modularity. Finally, in Section~\ref{sec.concl} we make some concluding remarks.


\section{Simulations}
\label{sec.sim}

For each $r =3,\ldots,12$ we generated ten instances of a random
$r$-regular graph on $10000$ nodes. The graphs were generated
using a variant of the configuration model which
was shown to converge to the uniform distribution in~\cite{randQuick}.
Modularity was optimised using two different methods, with  Table~\ref{tab:rreg} recording the averages for each method. 

Both methods start with each node in its own part (community).
The Louvain method~\cite{louvain} as implemented in~\cite{jeubcode} considers the nodes in turn, and
reshuffles a node into a different part if that increases the modularity (choosing a part which leads to the greatest increase).
It then forms a weighted reduced graph with a node for each part, and the process is repeated on the reduced graph.
The other method, `Reshuffle', follows Algorithm~1 of~\cite{martelot}. It has the same first phase. The second phase considers each part and merges it with a different part if that increases
the modularity {(again, choosing the part which leads to the greatest increase)}.   It then returns to the node shuffling phase, with the same nodes (we do not form a reduced graph, which would freeze earlier decisions).
There are no guarantees on the performance of these  modularity optimising heuristics.

In Figure~\ref{fig.qsim} we mark with an `X' the average value of the larger (better) of the output values of the two algorithms (which, with figures rounded to 3 decimal places as here, is the same as the larger of the averages), together with the
theoretical interval for the modularity given in Theorem~\ref{thm.kwus}.


\section{Proofs for 2-regular graphs}
\label{sec.2reg}

We first prove Proposition~\ref{prop.2reg} concerning the modularity $\q(G_{n,2})$ of a random 2-regular graph, using two preliminary lemmas; and then give the longer proof of Proposition~\ref{prop.minq}, which concerns $q_2^-(n)$ and the least modular 2-regular graphs. 
Finally we prove Proposition~\ref{prop.2regplus} on $q_2^+(n)$: it turns out to be easier to prove this result directly than to deduce it from Proposition~\ref{prop.maxqr} (because of the `$n$ sufficiently large' qualification in the latter result).
%
\begin{lemma} \label{lem.smalldiff}
Let $G=(V,E)$ and $G'=(V,E')$ be two graphs each with vertex set $V$ and $m$ edges, and with the same vertex degrees.  Then
\[ | \q(G)-\q(G')| \leq |E \Delta E'|/(2m).\] 
\end{lemma}
\begin{proof}
Let $\cA$ be a partition of $V$.  Then $q_{\cA}^E(G) - q_{\cA}^E(G') \leq |E \setminus E'|/m = |E \Delta E'|/(2m)$ and $q_{\cA}^D(G) = q_{\cA}^D(G')$, so $q_{\cA}(G) - q_{\cA}(G') \leq |E \Delta E'|/(2m)$.  Hence $q_{\cA}(G) - \q(G') \leq |E \Delta E'|/(2m)$.  But this holds for each partition $\cA$, so $\q(G) - \q(G') \leq |E \Delta E'|/(2m)$; and the lemma follows.
\end{proof}

\begin{lemma}\label{lem.fewcycles}
For $n$ sufficiently large, the expected number of cycles in a random 2-regular $n$-vertex graph is at most $\log n$.
\end{lemma}
\begin{proof}
We use the configuration model, see for example~\cite{jbook}.
Let $f(n)$ be the number of perfect matchings on a set of $n$ vertices.  Then $f(2n) = (2n-1)!! = (2n-1)(2n-3) \cdots 3 \cdot 1$.
  Let $M_n$ be a random 2-regular $n$-vertex multigraph.  For each integer $k$ with $3 \leq k \leq n$, the expected number $g(k,n)$ of $k$-cycles in $M_n$ equals
\begin{eqnarray*} 
\binom{n}{k} \frac{(k-1)!}{2} \, 2^{k-1}    \frac{f(n-2k)}{f(2n)}
  & = &
\frac{1}{4k} \, \frac{(n)_k 2^k}{\prod_{i=0}^{k-1} (2n-(2i+1))} \\
  & = &
\frac{1}{4k} \, \frac{\prod_{i=0}^{k-1} (2n-2i)}{\prod_{i=0}^{k-1} (2n-(2i+1))}.
\end{eqnarray*}
Hence, by comparing factors, 
$g(k,n) \leq \frac1{4k} \frac{2n}{2n-(2k-1)}$ (and $g(k,n) \geq \frac1{4k} $).  Since $\sum_{k >n/t} g(k,n) < t$, taking $t=\sqrt{\log n}$ say, we see that $\sum_{k=3}^{n} g(k,n) \leq (1/4 +o(1)) \log n$.
But the probability that $M_n$ is simple tends to $e^{-3/4}$, see for example Corollary 9.7 of~\cite{jbook}.  Hence the expected number of cycles in $G_2$ is at most $(e^{3/4}/4 +o(1)) \log n$, and the lemma follows.
\end{proof}
\smallskip

\begin{proof}[Proof of Proposition~\ref{prop.2reg}]
Let $\omega(n) \to \infty$ as $n \to \infty$.  By the last lemma and Markov's inequality, whp $G_2$ has at most $\omega(n) \log n$ cycles, and so there is a copy of $C_n$ such that the symmetric difference of the edge sets has size at most $4 \omega(n) \log n$. Thus by Lemma~\ref{lem.smalldiff}, whp $|\q(G_2)-\q(C_n)| \leq 2 \omega(n) \log n /n$.  Hence, by the result on $\q(C_n)$, whp $\q(G_2) = 1-2/\sqrt{n} + O(\omega(n) \log n /n)$.  Choosing $\omega(n)=o(\log n)$ completes the proof.
\end{proof}
\smallskip

\begin{proof}[Proof of Proposition~\ref{prop.minq}]
We may and shall restrict our attention to partitions into connected parts (as noted earlier). Suppose we are given a (large) integer $n$.
The $n$-\emph{cost} $f(t,{\cA})$ of a partition $\cA$ of $C_t$ into $k>1$ parts with $t_1,\ldots,t_k$ vertices is
\[ f(t,{\cA})= \tfrac{k}{n} + \tfrac{\sum_i t_i^2}{n^2}.\]
Here the $n$-cost refers to the contribution to $1-\q(G)$. For a given $k$, this cost is minimised when the $t_i$ are balanced (that is, differ by at most 1), so there is essentially just one partition to consider.

Let $F_k(t)$ be the $n$-cost of a balanced $k$-partition of $C_t$. 
Of course $F_1(t)= \tfrac{t^2}{n^2}$. 
Write $t$ as $ak+b$ with $0 \leq b \leq a-1$ (where $a=\lfloor t/k\rfloor$).  Then $t=(k-b)a +b(a+1)$, so for $2 \leq k \leq t$
\[ F_k(t)= \tfrac{k}{n} + \tfrac{(k-b)a^2 + b(a+1)^2}{n^2}. \]
Let $f_k(t)$ be defined for real $t$ with $0<t \leq n$, and be the natural approximation to $F_k(t)$, namely
 $f_1(t)= \tfrac{t^2}{n^2}$, and $f_k(t)= \tfrac{k}{n} + \tfrac{t^2}{kn^2}$ for $k \geq 2$.  Then $F_k(t) \geq f_k(t)$ by convexity.  Also
\begin{eqnarray}
\notag F_k(t)-f_k(t) & = &
\tfrac1{n^2} ((k-b)a^2+b(a+1)^2 - \tfrac{t^2}{k}\\
\label{eq.approxcost} & = &
\tfrac{b(k-b)}{ k n^2}  \;\; \leq \tfrac{k}{4n^2}.
\end{eqnarray}

The approximate `unit $n$-cost' $g_k(t)=f_k(t)/t$ is given by $g_1(t)= \tfrac{t}{n^2}$, and $g_k(t)= \tfrac{k}{nt} + \tfrac{t}{kn^2}$ for $k \geq 2$. Let $F_*(t)$ be the minimum over $k$ of $F_k(t)$, let $f_*(t)$ be the minimum over $k$ of $f_k(t)$, and let $g_*(t)$ be the minimum over $k$ of $g_k(t)$ (so $f_*(t)=g_*(t)\, t$). Let
$\gamma = \sqrt{2/3}+ \sqrt{3/2}=5/\sqrt{6}$. 
 
We shall establish five claims A,\ldots,E (with Claim C being used only to prove Claim D).
\smallskip

\noindent
\emph{Claim A.}
For all $0<t \leq n$ we have
\begin{equation} \label{eqn.h}
g_*(t) \leq g_*(\sqrt{6n})= \gamma n^{-3/2}.
\end{equation}%
\begin{proof}[Proof of Claim A] 
Consider the minima of the functions $g_k(t)$ for $k=2,3,\ldots$; and the crossings of the graphs of the functions $g_k(t)$ for $k=1,2,\ldots$.  We restrict attention to $t>0$. For $k \geq 2$, $g_k(t)$ is strictly convex and has minimum value $2n^{-3/2}$, achieved at $t=k \sqrt{n}$. The graphs of $g_1(t)$ and $g_2(t)$ meet when $t=2\sqrt{n}$ with common value $2n^{-3/2}$. For $k \geq 2$, the graphs of $g_k(t)$ and $g_{k+1}(t)$ meet at $t=\sqrt{k(k+1)n}$, with common value $\gamma_{k} n^{-3/2}$, where $\gamma_{k}= \sqrt{\tfrac{k}{k+1}}+\sqrt{\tfrac{k+1}{k}}$. Observe that $\max_{k\geq 2} \gamma_k=\gamma_2=\gamma$. 
Further, the curves $g_k(t)$ do not meet anywhere else (for $t>0$).
Hence
\[ g_*(t) \leq g_*(\sqrt{6n}) =g_2(\sqrt{6n})=g_3(\sqrt{6n})=\gamma n^{-3/2},\]
as required. 
\end{proof}
\smallskip

\noindent
\emph{Claim B.}
Let $\eps>0$.  Then there is a constant $c_0$ such that if $c_0 \sqrt{n} \leq t \leq n$ then (a) $F_*(t) \leq (1+\eps) 2n^{-3/2} t$, and (b) a balanced collection of about $t/\sqrt{6n}$ cycles of with combined number of vertices $t$ gives a total $n$-cost $ \geq (1-\eps) (5/\sqrt{6}) n^{-3/2} t$. 
\smallskip

\begin{proof}[Proof of Claim B]
(a).  Let $k= \lceil \tfrac{t}{\sqrt{n}}\rceil$.  Then the corresponding $n$-cost is at most
\[ \tfrac{k}{n} + \tfrac{k(\sqrt{n})^2}{n^2} = \tfrac{2k}{n} < \tfrac{2t}{n^{3/2}} + \tfrac{2}{n}= \tfrac{2t}{n^{3/2} } \, (1+ \tfrac{\sqrt{n}}{t}) \leq \tfrac{2t}{n^{3/2} } (1+ \eps) \, \]
if $t \geq (1/\eps) \sqrt{n}$.

(b).  There exists $\eta>0$ such that if $(1-\eta)\sqrt{6n} \leq t_i \leq (1+\eta)\sqrt{6n}$ then $g_*(t_i) \geq (1-\eps) \tfrac{5}{\sqrt{6}} n^{-3/2}$, and so $f_*(t_i) \geq (1-\eps) \tfrac{5}{\sqrt{6}} n^{-3/2}t_i$.  If $c_0$ is sufficiently large then each cycle in the balanced collection will have size $t_i$ in this range.  Hence the total $n$-cost will satisfy \[ \sum_i f_*(t_i) \geq (1-\eps) \tfrac{5}{\sqrt{6}} n^{-3/2} \sum_i t_i  =(1-\eps) \tfrac{5}{\sqrt{6}} n^{-3/2}t, \]
as required.
\end{proof}
\smallskip

\noindent
\emph{Claim C.}
If $t$ is an integer
$\leq 2 \sqrt{n}$, 
then $F_*(t)=F_1(t)$.
\smallskip

\begin{proof}[Proof of Claim C]
For $t \geq 3$ and $2 \leq k \leq t$,
\[ F_k(t)-F_1(t) \geq 
\tfrac{k}{n} +\tfrac{k(2t/k)^2}{4n^2} - \tfrac{t^2}{n^2} = \tfrac{k}{n} - \tfrac{t^2}{n^2} (1-\tfrac1{k}).\]
Thus
$ F_2(t)-F_1(t) \geq \tfrac{2}{n} - \tfrac{t^2}{2n^2} \geq 0$,
$ F_3(t)-F_1(t) \geq  \tfrac{3}{n} - \tfrac{2t^2}{3n^2} > 0$, and
for $k \geq 4$
\[ F_k(t)-F_1(t) > \tfrac{4}{n} -
\tfrac{t^2}{n^2} \geq 0. \]
This completes the proof.
\end{proof}
\smallskip

\noindent
\emph{Claim D.}
In an $n$-vertex 2-regular graph $G$ minimising $\q(G)$, at most one component $C_t$ has size $\leq \sqrt{n}$.
\smallskip

\begin{proof}[Proof of Claim D]
Suppose $G$ has two components $C_{t_1}$, $C_{t_2}$ with $t_1,t_2 \leq \sqrt{n}$.
Replace these two components by one component $C_{t_1+t_2}$.  By Claim C, $F_*(t_1+t_2)= F_1(t_1+t_2)$: hence the increase in $n$-cost is at least \[ F_1(t_1+t_2) - F_1(t_1)-F_1(t_2)
= \tfrac{t_1t_2}{2n^2} >0,\]
which completes the proof.
\end{proof}
\smallskip

\noindent
\emph{Claim E.}
Let $c_0 \geq \sqrt{2}$, and let $t \leq c_0 \sqrt{n}$.  Let $\cal A$ be a partition of $C_t$ minimising the $n$-cost.  Then $\cal A$ has $k \leq k_0= 1 + \sqrt{2} c_0$ parts.
\smallskip

\begin{proof}[Proof of Claim E]
Since $k_0 \geq 3$, we may assume that $k \geq 3$.
Suppose that $\cal A$ has two parts of sizes $t_1, t_2 < \sqrt{n/2}$.  Replace these two parts by a single part of size $t_1+t_2$ (where each part corresponds to a path).  The $n$-cost decreases by
\[\tfrac1{n} - \tfrac{(2(t_1+t_2))^2 - (2t_1)^2-(2t_2)^2}{4n^2}= 
\tfrac1{n} -\tfrac{2t_1t_2}{n^2} >0\]
since $t_1t_2 <n/2$, a contradiction.
Hence $\cal A$ has at most one part of size $< \sqrt{n/2}$.
It follows that the number of parts is less than $1+\tfrac{t}{\sqrt{n/2}}
 \leq k_0$.
\end{proof}

We can now use Claims B, D, E and A to prove the upper bound on $1-q_*(n)$.
Let $G=G_n$ minimise $\q(G)$ over $n$-vertex graphs.
By Claim B, with $\eps>0$ sufficiently small that $(1+\eps) \, 2 < (1-\eps) \, 5/\sqrt{6}$, 
each component of $G$ has size at most $c_0 \sqrt{n}$. 
By Claim D, $G$ has $s \leq 1+\sqrt{n}$ components.  Thus $G$ has components $C_{t_1},\ldots,C_{t_s}$ where $t_1+\ldots t_s=n$ and each $t_i \leq c_0 \sqrt{n}$.
Also, by Claim E, for each component~$C_t$ we need only to consider partitions with at most a constant $k_0$ parts.
Then by~(\ref{eq.approxcost}) and Claim~A
\begin{eqnarray*} 1-\q(G)
& = &
\sum_i F_*(t_i) \; \leq \; \sum_i f_*(t_i) \, + s \tfrac{k_0}{4n^2} \\
& = & 
\sum_i g_*(t_i) t_i + s \tfrac{k_0}{4n^2} \; \leq  \gamma n^{-3/2} \sum_i t_i + \tfrac{sk_0}{4n^2} \\
& = &
\gamma /\sqrt{n} + \tfrac{s k_0}{4n^2} = 
\gamma / \sqrt{n}  + O(n^{-3/2}).
\end{eqnarray*}

We have now seen that $1-q_*^{(2)}(n) \leq \gamma / \sqrt{n}  + O(n^{-3/2})$. To show the reverse inequality, consider a graph $G$ formed from $\lceil \sqrt{n/6}\rceil$ components, where each component is $C_{t_i}$ with $t_i = \sqrt{6n} + O(1)$.   
For $x=O(1)$, $g'_2(\sqrt{6n}+x) \sim \tfrac1{6n^2}$ and $g'_3(\sqrt{6n}+x) \sim -\tfrac1{6n^2}$.
 Thus for $k=2,3$ we have $g_k(t_i)= g_k(\sqrt{6n}) + O(n^{-2}) = \gamma n^{-3/2} + O(n^{-2})$; and so $g_*(t_i) = \gamma n^{-3/2} + O(n^{-2})$.  Hence the total $n$-cost  is 
\[ \sum_i g_*(t_i) t_i = (\gamma n^{-3/2} + O(n^{-2})) \sum_i t_i=
\gamma n^{-1/2} + O(n^{-1}),\] 
as required. 
This completes the proof of Proposition~\ref{prop.minq}.
\end{proof}

\begin{proof}[Proof of Proposition~\ref{prop.2regplus}]
The result is easy to check when $n=3,4,5$. Let $n \geq 6$ and let $G$ be a 2-regular $n$-vertex graph such that $\q(G)=q_2^+(n)$.  Clearly each component cycle of $G$ has order at most 5, since we could split a larger cycle to obtain a 2-regular $G'$ with $\q(G')>\q(G)$.  Further there cannot be a component $C_5$.  For, if there is another component $C_5$ we could replace two $C_5$'s by one $C_4$ and two $C_3$'s; if there is a component $C_4$ we could replace $C_5$ and $C_4$ by three $C_3$'s, and if there is a component $C_3$ we could replace $C_5$ and $C_3$ by two $C_4$'s: in each case we would strictly increase the modularity.  Thus the only possible components in $G$ are $C_3$ and $C_4$.  Further there can be at most two $C_4$'s, as we could replace three $C_4$'s by four $C_3$'s.  It now follows easily that the optimal configurations are as claimed.  Finally it is now easy to check that we have the correct formulae for $q_2^+(n)$.
\end{proof}


\section{Proofs for treewidth and maximum degree}
\label{sec.tw-proofs}

To prove Theorem~\ref{thm.tw} (the `treewidth lower bound') we need one preliminary lemma. Given a graph and a partition of its vertex set, a \emph{cross-edge} is an edge with its end vertices in different parts of the partition.

\begin{lemma}\label{lem.twcutdir}
  Let the graph $G$ have $m$ edges and maximum degree at most $d$, and let the set $E'\subset E(G)$ be such that the subgraph $H=G\backslash E'$ satisfies $\tw(H) \leq t$. Let $s$ satisfy $d<s \leq 2m-d$. Then by deleting from $H$ the edges incident with at most $t+1$ vertices, thus forming the subgraph $H'$, we can find a partition $V(G) =U_0 \cup \cdots \cup U_{k}$ with $k\geq 2$, no cross-edges in $H'$, and such that $\ds(U_0) \leq 2m-s$ and $\ds(U_j) <s$ for each $j=1,\ldots,k$.  (We allow $U_0=\emptyset$.)
  \end{lemma}
  
  \begin{figure}

\begin{center}
\scalebox{0.82}{
\begin{tikzpicture}
\node at (1,4.5) {$G=$};
\draw (2.5,4)-- (3,5) --(2,5) -- (2.5,4) -- (3.5,4); 
\draw (3,5) --(4,5)--(2.5,4); 
\draw[dashed] (2,5)-- (3.5,4) --(3,5); 

\node[circle,fill=orange, minimum width=.2cm] (ball) at (2,5) {};
\node[circle,fill=gray, minimum width=.2cm] (ball) at (2.5,4) {};
\node[circle,fill=green, minimum width=.2cm] (ball) at (3,5) {};
\node[circle,fill=red, minimum width=.2cm] (ball) at (3.5,4) {}; 
\node[circle,fill=blue, minimum width=.2cm] (ball) at (4,5) {};

\node at (1,2.5) {$H=$};
\draw (2.5,2)-- (3,3) --(2,3) -- (2.5,2) -- (3.5,2);
\draw (3,3) --(4,3)--(2.5,2); 
\node[circle,fill=orange, minimum width=.2cm] (ball) at (2,3) {};
\node[circle,fill=gray, minimum width=.2cm] (ball) at (2.5,2) {};
\node[circle,fill=green, minimum width=.2cm] (ball) at (3,3) {};
\node[circle,fill=red, minimum width=.2cm] (ball) at (3.5,2) {}; 
\node[circle,fill=blue, minimum width=.2cm] (ball) at (4,3) {};


\begin{scope}[thin]
  \draw (7,4.5)-- (7,3)node[draw=none,fill=none,font=\small,midway,right] {16};
  \draw (7,3)-- (7,1.5)node[draw=none,fill=none,font=\small,midway,right] {16};
  \draw[blue, very thick] (7,1.5)-- (7,0)node[draw=none,fill=none,font=\small,midway,right] {13};

  \draw (8.6,-0.2)-- (7,0)node[draw=none,fill=none,font=\small,midway,above] {10}; 
  \draw (10.2,-0.4)-- (8.6,-0.2)node[draw=none,fill=none,font=\small,midway,above] {6}; 
  \draw (11.8,-0.6)-- (10.2,-0.4)node[draw=none,fill=none,font=\small,midway,above] {2}; 


  \draw (5.4,-0.2)-- (7,0)node[draw=none,fill=none,font=\small,midway,above] {7}; 
  \draw (3.8,-0.4)-- (5.4,-0.2)node[draw=none,fill=none,font=\small,midway,above] {7}; 
  \draw (2.2,-0.6)-- (3.8,-0.4)node[draw=none,fill=none,font=\small,midway,above] {3}; 
\end{scope}

\begin{scope}[thick, every node/.style={sloped,allow upside down}]
  \draw (7,4.5)-- node {\midarrow} (7,3);
  \draw (7,3)-- node {\midarrow} (7,1.5);
  \draw[blue, very thick] (7,1.5)-- node {\midarrow} (7,0);

\draw (8.6,-0.2)-- node {\midarrow} (7,0); 
\draw (10.2,-0.4)-- node {\midarrow} (8.6,-0.2); 
\draw (11.8,-0.6)-- node {\midarrow} (10.2,-0.4); 
  
  \draw (5.4,-0.2)-- node {\midarrow} (7,0); 
  \draw (3.8,-0.4)-- node {\midarrow} (5.4,-0.2); 
  \draw (2.2,-0.6)-- node {\midarrow} (3.8,-0.4); 
\end{scope}

\filldraw[dotted, color=blue, fill=white, fill opacity=0, thick](6.8,-0.92) ellipse (6 and 1.9){};
\node[blue] at (6.8,0.7) {$e$}; 
\node[blue] at (4,-1.8) {$T_e$};

\filldraw[color=red, fill=white, fill opacity=1, very thick](7,4.5) ellipse (0.5 and 0.4){};
\node[circle,fill=orange, minimum width=.2cm] (ball) at (7,4.5) {};
\node[red] at (7.8,4.8) {\small root}; 

\filldraw[color=red, fill=white, fill opacity=1, very thick](7,3) ellipse (0.5 and 0.4){}; 
\node[circle,fill=green, minimum width=.2cm] (ball) at (6.77,3) {};
\node[circle,fill=orange, minimum width=.2cm] (ball) at (7.23,2.95) {};

\filldraw[color=red, fill=white, fill opacity=1, very thick](7,1.5) ellipse (0.5 and 0.4){}; 
\node[circle,fill=gray, minimum width=.2cm] (ball) at (7,1.68) {};
\node[circle,fill=green, minimum width=.2cm] (ball) at (7.3,1.5) {};
\node[circle,fill=orange, minimum width=.2cm] (ball) at (6.75,1.4) {};

\filldraw[color=red, fill=white, very thick](7,0) ellipse (0.5 and 0.4){}; 
\node[circle,fill=gray, minimum width=.2cm] (ball) at (6.77,-0.1) {};
\node[circle,fill=green, minimum width=.2cm] (ball) at (7.23,0.1) {};
\node[blue] at (7,-0.6) {\small $i$};

\filldraw[color=red, fill=white, fill opacity=1, very thick](8.6,-0.2) ellipse (0.5 and 0.4){}; 
\node[circle,fill=green, minimum width=.2cm] (ball) at (8.4,-0.08) {};
\node[circle,fill=blue, minimum width=.2cm] (ball) at (8.8,-0.2) {};
\node[circle,fill=gray, minimum width=.2cm] (ball) at (8.5,-0.4) {};
\node[blue] at (8.5,-0.8) {\small $j_2$};

\filldraw[color=red, fill=white, fill opacity=1, very thick](10.2,-0.4) ellipse (0.5 and 0.4){}; 
\node[circle,fill=green, minimum width=.2cm] (ball) at (10,-0.4) {};
\node[circle,fill=blue, minimum width=.2cm] (ball) at (10.4,-0.5) {};

\filldraw[color=red, fill=white, fill opacity=1, very thick](11.8,-0.6) ellipse (0.5 and 0.4){}; 
\node[circle,fill=blue, minimum width=.2cm] (ball) at (11.8,-0.6) {};

\filldraw[color=red, fill=white, fill opacity=1, very thick](5.4,-0.2) ellipse (0.5 and 0.4){}; 
\node[circle,fill=gray, minimum width=.2cm] (ball) at (5.4,-0.2) {};
\node[blue] at (6,-0.5) {\small $j_1$};

\filldraw[color=red, fill=white, fill opacity=1, very thick](3.8,-0.4) ellipse (0.5 and 0.4){}; 
\node[circle,fill=gray, minimum width=.2cm] (ball) at (4.03,-0.35) {};
\node[circle,fill=red, minimum width=.2cm] (ball) at (3.57,-0.45) {};

\node[black] at (10,4.5) { $V_e=\{$};
\node[black] at (12.15,4.5) { $\}$};
\node[circle,fill=gray, minimum width=.2cm] (ball) at (10.7,4.5) {};
\node[circle,fill=red, minimum width=.2cm] (ball) at (11.1,4.5) {};
\node[circle,fill=green, minimum width=.2cm] (ball) at (11.5,4.5) {};
\node[circle,fill=blue, minimum width=.2cm] (ball) at (11.9,4.5) {};

\node[black] at (10,4) { $U_0=\{$};
\node[black] at (10.95,4) { $\}$};
\node[circle,fill=orange, minimum width=.2cm] (ball) at (10.7,4) {};

\node[black] at (10.6,3.5) { $U_1=V_{i{j_1}}=\{$};
\node[black] at (12.55,3.5) { $\}$};
\node[circle,fill=red, minimum width=.2cm] (ball) at (11.9,3.5) {};
\node[circle,fill=gray, minimum width=.2cm] (ball) at (12.3,3.5) {};

\node[black] at (10.9,3) { $U_2=V_{i{j_2}}\backslash U_1=\{$};
\node[black] at (13.12,3) { $\}$};
\node[circle,fill=green, minimum width=.2cm] (ball) at (12.48,3) {};
\node[circle,fill=blue, minimum width=.2cm] (ball) at (12.88,3) {};

\filldraw[color=red, fill=white, fill opacity=1, very thick](2.2,-0.6) ellipse (0.5 and 0.4){}; 
\node[circle,fill=red, minimum width=.2cm] (ball) at (2.2,-0.6) {};
\end{tikzpicture}
}
\end{center}

\caption[Treewidth illustration.]{\small An illustration of the construction in the proof of Lemma~\ref{lem.twcutdir} applied to a toy graph~$G$ with removed edge set $E'$ (dashed) and $s=12$. Graph $G$ has treewidth 3 but after removing the dashed edges graph $H$ has treewidth 2. A tree-decomposition for $H$ is shown and the leaf node at the top chosen to be the root. For each edge~$h$ in the tree-decomposition the number $w(V_h)$ is shown, and the edge oriented toward the root if $w(V_h)<12$. The rooted tree-decomposition and threshold $s$ define node $i$, edge $e$, component $T_e$, and partition $V=U_0 \cup U_1 \cup U_2$ as shown.}
\label{fig.twcut}
\end{figure}
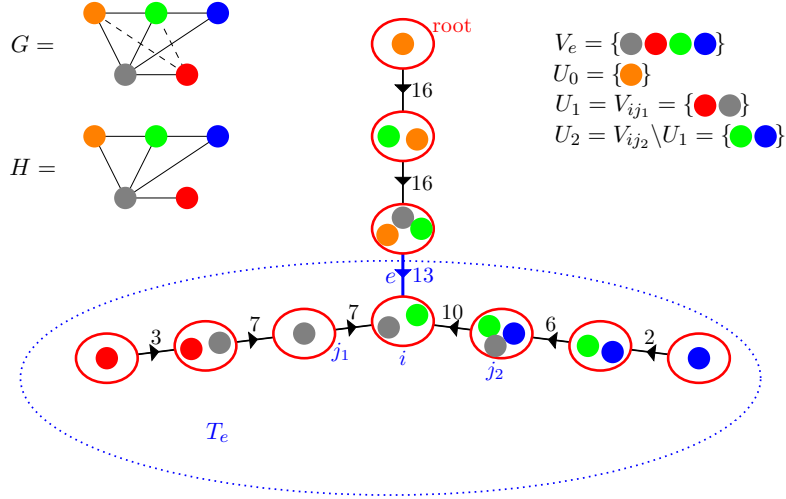
  
\begin{proof} 
The first step is to introduce a weight function which remembers information about the edges in $E'$: for a vertex $v\in H$ let $w(v)=\deg_G(v)$, and for a set $U$ of vertices let $w(U)=\sum_{v\in U}w(v)$.
 The proof will take a tree-decomposition of $H$, choose one bag $X_i$, and delete all edges of $H$ incident to the vertices in $X_i$.
 
We want a suitably well-behaved tree-decomposition, see for example~\cite{karb}. 
 It is well known that (by adding nodes if necessary) we can guarantee a tree decomposition $T$ of width at most $t$ such that if $ij$ is an edge of~$T$ then the symmetric difference $X_i \triangle X_j$ has exactly one element. If distinct nodes $i$ and $j$ have the same bag (that is, if $X_i=X_j$) and have a common neighbour, then we can replace the nodes $i$ and $j$ by a single new node with the same bag: thus we may assume that if nodes $i$ and $j$ have a common neighbour in $T$ then $X_i\neq X_j$.  Finally, again by adding nodes if necessary, we may assume that each leaf~$i$ of~$T$ has bag $X_i$ of size~1. Fix such a tree decomposition, and fix a leaf to be the root vertex. 

Recall that deleting any edge in a tree leaves exactly two connected components. Let us recall also one simple standard fact about tree-decompositions (as introduced in Subsection 1.3).

\noindent
\begin{fact} \label{fact1}
If $e= ij$ is an edge of the tree $T$, $k$ and $k'$ are nodes of $T$ in different components of $T \backslash e$, 
and vertices $v \in X_{k}$ and $v' \in X_{k'}$ are adjacent in $G$, then at least one of $v$, $v'$ is in $X_i \cap X_j$.
\end{fact}

For any edge $e$ in $T$ let $T_e$ denote the non-root component of $T \backslash e$, and let $V_e$ be the set of vertices contained in the bags of $T_e$.
 
If $w(V_e) <s$, then orient $e$ toward the root vertex, otherwise orient $e$ away from the root vertex. (See Figure~\ref{fig.twcut} for an illustration.)

 At least one node in $T$ has out-degree zero: fix such a node~$i$. Notice that~$i$ is not the root (since $s \leq 2m-d$), and~$i$ is not a leaf (since then $|X_i|=1$ and so $w(X_i) \leq d<s$). We delete the edges of $H$ incident with the vertices in the bag $X_i$.  Thus we delete at most $(t+1)d$ edges from $H$ to form $H'$.  Let~$e$ be the edge incident with node~$i$ which lies on the path from the root vertex to node~$i$.  Let $U_0= V(G) \setminus V_e$.  Since $w(V_e) \geq s$ we have $w(U_0) \leq 2m-s$.

Since~$i$ is not a leaf in $T$, other than its neighbour along edge $e$, $i$ has neighbours $j_1, \ldots, j_h$ for some $h \geq 1$. 
Suppose first that $h=1$, so there is exactly one such neighbour $j_1$ (not along the edge $e$). 
Since the edge $ij_1$ is oriented towards $i$, we have
$w(V_{i j_1})<s \leq w(V_e)$, and so we cannot have $X_{j_1} \supseteq X_i$: hence $X_{j_1}= X_i \setminus \{v\}$ for some $v \in X_i$.
Let $U_1=V_{ij_1}$ and $U_2 = X_i \setminus U_1=\{v\}$.
Then $w(U_1)<s$ and $w(U_2) =w(v) \leq d < s$, and 
$V(G)$ is partitioned into $U_0 \cup U_1 \cup U_2$. Further, by Fact~\ref{fact1}, any cross-edges in $H$ must be incident to vertices in $X_i$, and so there are no cross-edges in $H'$.

Now suppose that $h \geq 2$, so node $i$ has multiple neighbours.  Let  $U_1=V_{ij_1}, U_2=V_{ij_2}\setminus U_1,$ $\ldots, U_{h} = V_{ij_h} \setminus (U_1 \cup \ldots \cup U_{h-1}) $.
As before, the orientations of the edges incident with node~$i$ shows that $w(U_j)<s$ for each~$j =1,\ldots,h$.  Discard any empty sets amongst these sets~$U_j$. 
Finally, note that $X_{j_1}\cup X_{j_2}\supseteq X_i$, since if neither of $X_{j_1}$ and $X_{j_2}$ contains $X_i$ then $X_{j_1}=X_i\setminus \{v_1\}$ and $X_{j_2}=X_i\setminus \{v_2\}$ for some $v_1\neq v_2$ in $X_i$.  Hence $U_1\cup U_2\supseteq X_i$ and so $\cup_{j=0}^{h} U_j=V(G)$. Further, as before, by Fact~\ref{fact1} any cross-edges in $H$ must be incident to vertices in $X_i$, and so there are no cross-edges in $H'$.
\end{proof}
\smallskip
\begin{proof}[Proof of Theorem~\ref{thm.tw}] \;\;
Write $d$ for the maximum degree $\Delta$ of $G$ (note that $G$ will shrink during the proof but $d$ stays unchanged). Since $\q(G) \geq 0$ for any graph $G$ we need to consider only the case where $m \geq 4(t+1)d$. 
Let $s= 2((t+1)dm)^{\frac12}$.  Note that $s \geq 4(t+1)d$. 

Set $\tilde{G}=G$ and $\tilde{m}=e(\tilde{G})$. Observe that $s>d$, indeed $s \geq 2d$. As long as $2\tilde{m} \geq s+d$ we use the last lemma repeatedly to `break off parts' $U_1, U_2, \ldots$ and replace $\tilde{G}$ by its induced subgraph on $U_0$, where $\ds(U_0)\leq 2\tilde{m}-s$.  Suppose that we stop after $j$ steps, with $2\tilde{m}=x$ where $0 \leq x <s+d$. Since at each step the degree sum of $\tilde{G}$ decreases by at least $s$, we must have $x \leq 2m - js$, so $js \leq 2m-x$. At this stage we have lost at most $j(t+1) d$ edges, and each of the parts `broken off' from $G$ has degree sum $< s$.
We {\bf claim} that we can refine the current partition of $V(G)$ to a partition $\cA$ such that each part has degree sum $<s$ and the number of cross-edges in $G\setminus E'$ (edges lost) is at most $\frac{2m}{s} (t+1)d$. There are two cases.
  
Suppose first that $x<s$: then we have finished in $j$ steps, $j \leq 2m/s$, and we have lost at most $\frac{2m}{s} (t+1)d$ edges in $G \setminus E'$, as required.
Suppose instead that $x \geq s$.  Then $js \leq 2m-x \leq 2m-s$, so $j +1 \leq 2m/s$.  We take one more step, with reduced threshold $s'=s-d$.  Note that $d< s' \leq 2\tilde{m} -d$ (where $2\tilde{m}=x$).
Thus we can apply Lemma~\ref{lem.twcutdir} with the value $s'$, to complete the proof of the claim, since
  \[ 2\tilde{m}-s' < s+d-(s-d) = 2d \leq s, \] 
and so we stop after $j+1$ steps.  
  
Now 
$q^E_\cA(G) \geq 1 -  |E'|/m - 2(t+1)d/s$.
Also $0\leq x_i < s$ and $\sum_i x_i \leq 2m$ together imply $\sum_i x_i^2 < 2ms$; and so 
\begin{equation}\label{eq.son2m}
q_\cA^D(G) \leq \frac{2ms}{4m^2}= \frac{s}{2m}.
\end{equation}
Hence, by our choice of $s$,
\[ 1-|E'|/m - q_\cA(G) \leq \tfrac{2(t+1)d}{s} + \tfrac{s}{2m} = 2 \left( \tfrac{(t+1)d}{m}\right)^{\frac12},\]
which completes the proof.
 \end{proof}
\medskip


We may deduce Proposition~\ref{cor.allrreg} quickly from Theorem~\ref{thm.tw}.

\begin{proof}[Proof of Proposition~\ref{cor.allrreg}]\label{cor.proof} 
For each connected component $H$ of $G$, do the following. In $H$ choose a spanning tree together with one extra edge (observe that $H$ is not a tree since $r \geq 2$), and let $E'_H$ be the set of edges not chosen. 

Each unicyclic graph has treewidth 2, 
so $\tw(H\setminus E'_H)=2$. Define $E'=\cup_H E'_H$, and note that $\tw(G\setminus E')=2$ and $|E'|=m-n$, where $G$ has $m=r n/2$ edges.  Hence by Theorem~\ref{thm.tw}
\[ \q(G) \geq 1-2\big(\frac{3r}{m}\big)^{\frac12}- (1-\frac{n}{m}) = \frac{2}{r} - 2 \big(\frac{6}{n}\big)^{\frac12}\] 
as required.
\end{proof}


\section{Proofs for minimum and random modularity, $q^-_r(n)$ and $\q(\Gr)$} \label{sec.minrand-proofs}

In this section we prove Theorem~\ref{thm.rlarge-min}, giving a lower bound on $q_r^-(n)$; and Theorems~\ref{thm.kwus} and~\ref{thm.rlarge-random} giving lower and upper bounds which hold whp on $\q(\Gr)$, the former for small $r$ ($r=3,\ldots,12$) and the latter for large $r$.
The lower bound proofs are all centred around bisection width, and the upper bound proofs around edge-expansion.  It is thus natural and convenient to prove all the lower bounds in the theorems first and then the upper bounds.

\subsection{Bisection width and lower bounds on modularity}
\label{subsec.bw-ubs}

Define the \emph{bisection width} $\bw(G)$ of an $n$-vertex graph $G$ to be
\[\bw(G) = \min_{|U| = \lfloor \frac{n}{2} \rfloor} e(U, \bar{U})\]
where the minimum is over all sets $U$ of $\lfloor \frac{n}{2} \rfloor$ vertices, and 
$\bar{U}$ denotes $V(G)\setminus U$. A corresponding minimising partition shows that, 
for an $r$-regular graph $G$,
\begin{equation}\label{eqn.qbw}
 \q(G) \geq \frac{1}{2} - \frac{2\bw(G)}{rn} -  \frac1{2n^2}
\end{equation}
where we do not need the last (small) term if $n$ is even.

\begin{proof}[Proof of  Theorem~\ref{thm.rlarge-min}] \hspace{.1in}

Now consider general $r$.  By Theorem 1.1 of Alon~\cite{alon1997edge}, 
there is a constant $c>0$ such that, for all $r$ and all sufficiently large $n$, all $n$-vertex $r$-regular graphs $G$ satisfy
\[ \bw(G)/n \leq r/4 - c \sqrt{r}. \] 
Hence, for each such graph $G$, 
using also~(\ref{eqn.qbw}) we have
\[ \q(G) \geq 2c/ \sqrt{r} -  \frac1{2n^2} \geq  c/ \sqrt{r} \] 
 for $n$ sufficiently large.
\end{proof}

\begin{proof}[Proof of lower bounds in Theorem~\ref{thm.kwus}]
\hspace{.1in}
 
Lower bounds for $r=3$ to $8$ are directly from the about $2/r$ bound in Proposition~\ref{cor.allrreg}, so consider larger $r$. It was shown in~\cite{diaz2007bounds} that whp the bisection width of a random 12-regular graph is at most 1.823n. By~(\ref{eqn.qbw}), this implies that whp $\q(G_{n,12})>0.196$, as given in Theorem~\ref{thm.kwus} (Table I). Similar calculations apply for $r=9,10,11$ which have bisection widths at most 1.2317, 1.4278, 1.624 times $n$ respectively~\cite{diaz2007bounds}. (Currently known results on bisection width do not improve on the $2/r$ lower bound from Proposition~\ref{cor.allrreg}.)
\end{proof}

\begin{proof}[Proof of lower bound in Theorem~\ref{thm.rlarge-random}]
\hspace{.1in}

For large $r$, Dembo et al.\ \cite{dembo} [Theorem~1.5] show that whp $\bw(\Gr)/n = r/4-c'\sqrt{r}/2+o(\sqrt{r})$, where the $o(\sqrt{r})$ error term is as $r\to \infty$ and $c'=0.76321\pm 0.00003$. By~(\ref{eqn.qbw}) this provides the lower bound in Theorem~\ref{thm.rlarge-random}.
\end{proof}


\subsection{Edge-expansion and upper bounds on modularity}
\label{subsec.ee-lbs}
In this subsection we introduce graph parameters $\beta(G)$, $\lambda(G)$, $i_u(G)$, $\alpha(G)$ and $\beta'(G)$ related to edge-expansion; and give two lemmas concerning them, in preparation for proving the upper bounds on $\q(G)$ in Theorems~\ref{thm.kwus} and~\ref{thm.rlarge-random}.

First we introduce a useful quantity $\beta(G)$ for a regular graph $G$, related to edge expansion. For a non-empty set $S$ of vertices in a graph $G$, let $\bar{d}(S)$ denote the average degree of the induced subgraph on $S$, so 
$\bar{d}(S) = 2e(S)/|S|$.  Now let $r \geq 2$ and suppose that $G$ is $r$-regular and has $n$ vertices.  Let 
\[\beta(G) = \max_{S \neq \emptyset}  \left\{\frac{\bar{d}(S)}{r} - \frac{|S|}{n} \right\}\]
where the maximum is over all non-empty sets $S$ of vertices. 
Also, given an $n$-vertex graph~$G$ such that the adjacency matrix has eigenvalues $\lambda_ 1 \geq .. \geq \lambda_n$,  let
\[ \lambda(G) = \max_{i>1} |\lambda_i|  \;\;\; (\: =\max\{|\lambda_2|,|\lambda_n|\} \: ).\]

The following lemma is the key to our upper bounds on $\q$ for random regular graphs. A more general result implying that $\q(G) \leq \lambda/r$ appeared earlier in~\cite{van2010spectral} statement (18), phrased in terms of the `modularity matrix' of $G$, which has largest eigenvalue equal to $\lambda(G)$ when $G$ is regular. We introduce $\beta(G)$ and give a short and straightforward proof.

\begin{lemma}\label{lem.ub1}
Let $G$ be an $r$-regular graph, let $\beta=\beta(G)$, and let $\lambda= \lambda(G)$. Then
\[ \q(G)  \leq \beta \leq \lambda/r.\]
\end{lemma}

\begin{proof}
Let $G$ have $n$ vertices, let $S$ be a non-empty set of vertices, and let $u=|S|/n$.
By Corollary 9.2.6 of Alon and Spencer~\cite{alon2000probabilistic} (see also Lemma 2.3 of Alon and Chung~\cite{alon1988explicit})
\[ |e(S)-\frac12 r u^2n|
\leq \frac12 \lambda un.\]
Hence
\[ \left| \frac{\bar{d}(S)}{r} - u \right| 
= \frac{|e(S)-\frac12 r u^2n|}{\frac12 run} \leq 
 \lambda/r;\]
and so $\beta \leq \lambda/r$.
Now consider any partition $\cA=\{A_1, \ldots, A_k\}$ of $V(G)$. Letting $u_j=|A_j| / n$, we have
\begin{eqnarray*}
 q_\cA(G)	
			&=& \sum_j  \left(\frac{2e(A_j)}{r n} - \frac{(u_j r n)^2}{(rn)^2} \right) \\
			&=&  \sum_j u_j \left( \frac{\bar{d}(A_j)}{r} - u_j\right)\\
			& \leq &  \beta \, \sum_j u_j   \;\; = \; \beta.
\end{eqnarray*}
Hence $\q(G) \leq \beta$, as required.
\end{proof}
\smallskip

Now we relate $\beta(G)$ to edge expansion and the quantity $\alpha(G)$ defined below, so that we can use calculations from~\cite{KolWorm}.
Following the notation of~\cite{KolWorm}, for $0<u\leq \frac{1}{2}$ we define the $u$-\emph{edge-expansion} $i_u(G)$ of an $n$-vertex graph $G$ by setting  
$$i_u(G)= \min_{0<|S|\leq un} \frac{e(S, \bar{S})}{|S|}$$
where the minimum is over non-empty sets $S$ of at most $un$ vertices
(and the value is taken to be $\infty$ if $un<1$).
Observe that $i_{1/2}(G)$ is the usual edge expansion or isoperimetric number of $G$.  Also, set
\[ \alpha(G) = \min_{0<u \leq \frac12} \{u+i_u(G)/r \}.\]
It is easy to see that we can also write $\alpha(G)$ as
\[ \alpha(G) = \min_{0<|S| \leq \frac12 n} \left\{\frac{|S|}{n}+ \frac{e(S, \bar{S})}{r|S|} \right\}.\]
Now consider a quantity $\beta'$ like $\beta$ but which at first sight might be smaller: let 
\[\beta'(G) = \max_{0<|S| \leq \frac12 n}  \left\{\frac{\bar{d}(S)}{r} - \frac{|S|}{n} \right\}.\]

\begin{lemma} \label{lem.alphabeta}
For each regular graph $G$,
\[ \beta(G) = \beta'(G) = 1-\alpha(G).\]
\end{lemma}
\begin{proof}
Note first that
$ \tfrac{\bar{d}(S)}{r} = \tfrac{r|S| - e(S,\bar{S})}{r|S|} = 1- \tfrac{e(S,\bar{S})}{r|S|}$, so
\[  \frac{\bar{d}(S)}{r}  -\frac{|S|}{n} =  1- \left(\frac{|S|}{n} + \frac{e(S,\bar{S})}{r|S|} \right).\]
It follows directly that $\beta'(G) = 1- \alpha(G)$.

Now write $\beta'$ for $\beta'(G)$: we shall show that $\beta'= \beta(G)$.
Let $S$ be a set of vertices with $|S|/n=u > \frac12$.
We must show that  $\frac{\bar{d}(S)}{r} -u \leq \beta'$. 
Since $2e(S)=r|S| - e(S,\bar{S})$ and similarly $2e(\bar{S})=r(n-|S|) - e(S,\bar{S})$, we have
\[ 2e(S) = 2e(\bar{S}) +r un-r(1\!-\!u)n = 
2e(\bar{S}) + (2u\!-\!1)rn. \]
Also $\frac{2e(\bar{S})}{r(1- u) n} - (1-u) \leq \beta'$, so
$2e(\bar{S}) \leq r (1-u)n \, (1-u+\beta')$. Hence
\begin{eqnarray*}
\frac{\bar{d}(S)}{r} -u \;
&=& \;
\frac{2e(S)}{r un} - u \;\; = \;\; 
  \frac{2e(\bar{S})+(2u-1)rn}{r un} -u\\
& \leq &
   \frac{r(1-u)n (1-u+\beta') +(2u-1)rn}{run} -u\\
   &=&
    \frac{(1-u) \beta'}{u} \; \leq \beta'
\end{eqnarray*}
since $u \geq \frac12$. This completes the proof.
\end{proof}

\begin{proof}[Proof of upper bounds in Theorem~\ref{thm.kwus}] \hspace{.1in}

Fix an integer $r \geq 3$.  By Lemmas~\ref{lem.ub1} and~\ref{lem.alphabeta}, $1- \q(G) \geq \alpha(G) = \min_{u\in (0,1/2]} f(u)$ where $f(u)=u+ i_u(G)/r$.  Thus we want a lower bound on $\alpha(G)$.  Let $0 \leq u_0<u_1\leq 1/2$.  Since $i_u(G)$ is non-increasing in $u$, for $u\in (u_0,u_1]$ we have
\[ f(u) = u+ \frac{i_u(G)}{r}  \geq u + \frac{i_{u_1}(G)}{r} > f(u_1)-|u_1-u_0|. \]
Fix $\eps >0$, and let $n=\lceil 1/\eps \rceil$.
Then by considering the intervals $((i-1)/2n, i/2n]$ for $i=1,\ldots n$, we see that
\[ \alpha(G) > \min_{i=1,\ldots,n} f(i/2n) \; - \eps/2,\]
and so we need lower bound at most $n$ values $i_u(G)$.


Now consider a random $r$-regular graph $\Gr$. For each particular $u\in (0,1/2]$, by Theorem~1.3 in~\cite{KolWorm} and the discussion in Section~7 of that paper, we see that whp $i_u(\Gr)\geq y/u$ for each $0<y< ru(1-u)$
with $\hat{f}_r(u,y)<0$, where \[\hat{f}_r(u,y)=\log \frac{r^{r/2}u^{(r-1)u}(1-u)^{(r-1)(u-1)}}{y^y(ru-y)^{(ru-y)/2}(r-ru+y)^{(r-ru+y)/2} }.\]

(It is known~\cite{KolWorm} that $\hat{f}_r(u,y)$ is strictly concave in $y$, positive at $y= ru(1-u)$ and negative for sufficiently small $y>0$.)  For a random cubic graph $G_{n,3}$, by finding appropriate values $y$, we deduce that whp $\alpha(G_{n,3}) > 0.196$, and so whp $\q(G_{n,3})< 0.804$. Repeating for other values of $r$ yields the upper bounds given in Table~\ref{tab:rreg}. This completes the proof of Theorem~\ref{thm.kwus} (which first appeared in the thesis of the second author~\cite{thesis}).
\end{proof}

\begin{proof}[Proof of upper bound in Theorem~\ref{thm.rlarge-random}]
\hspace{.1in}

Let $\Gr$ be a random $r$-regular graph (with $r$ fixed). Friedman~\cite{friedman2008proof} showed that, for each $\eps>0$, whp $\lambda(\Gr) \leq 2\sqrt{r-1} + \eps$. Thus, by taking $\eps< 2\sqrt{r}-2\sqrt{r-1}$, we see that whp $\lambda(\Gr) < 2\sqrt{r}$. Hence by Lemma~\ref{lem.ub1}, whp $q^*(\Gr) < 2/ \sqrt{r}$.
This completes the proof of Theorem~\ref{thm.rlarge-random}.
\end{proof}


\section{Proofs for maximum modularity $q_r^+(n)$}
\label{sec.maxmod-proofs}

In this section our main task is to prove Proposition~\ref{prop.maxqr}. But first let us deal with Proposition~\ref{prop.connub}, which has a short and easy proof, see~\cite{FortBart2008} equation (11), or see for example~\cite{dasgupta2013complexity} Lemma 2.1.  We prove it here as we want the exact result.

\begin{proof}[Proof of Proposition~\ref{prop.connub}] \hspace{.1in}

Let the graph $G$ have $m \geq 1$ edges (we do not yet assume that $G$ is connected).  Consider a partition $\cA=\{A_1, \ldots, A_k\}$ of $V(G)$ into $k \geq 2$ parts, where $A_j$ has degree sum $d_j$ (and so $\, \sum_j d_j =2m$).  Then 
\begin{equation} \label{eqn.degtaxreg}
 q_{\cA}^{D}(G) = \frac1{4m^2} \sum_j d_j^2 \geq \frac1{k}
 \end{equation}
since $(1/k) \sum_j d_j^2 \geq (2m/k)^2$ by convexity.
If $G$ is connected, then there must be at least $k-1$ cross-edges; and then, since $k/m + 1/k \geq 1/2\sqrt{m}$, 
\[ q_{\cA}(G) \leq 1- \frac{k-1}{m} - \frac1{k} \leq 1+\frac1{m} - \frac2{\sqrt{m}}. \] 
Similarly, if $G$ is 2-edge-connected, there must be at least $k$ cross-edges, and 
\[ q_{\cA}(G) \leq 1- \frac{k}{m} - \frac1{k} \leq 1- \frac2{\sqrt{m}}, \]
which completes the proof. 
\end{proof}

To prepare for the proof of Proposition~\ref{prop.maxqr}, we first give an equivalent expression for the modularity of a regular graph, and give a preliminary lemma.
 Observe that for $r$-regular $n$-vertex graphs $G$ we have
 \begin{equation}\label{eq.modis}
q_\cA(G)= 1-\sum_{A \in \cA} \frac{|A|}{n} \Big(\frac{e(A, \bar{A})}{r|A|} + \frac{|A|}{n} \Big)
\end{equation}
where $\bar{A}$ denotes $V(G)\backslash A$. Here the first term in the sum double counts the edges between the parts (and divides by twice the number of edges), and the second term takes care of the degree tax, which is now simply a function of the part sizes since $G$ is regular.

For $\emptyset \neq A \subseteq V(G)$ let $f_r(A,G) = e(A, \bar{A})/(r|A|)+ |A|/n$ and write $G[A]$ for the subgraph of $G$ induced by $A$. 
We shall see that the unique minimiser of $f$ over $r$-regular graphs is the graph $K_{r+1}$. More fully, assuming that $n \geq r+1$ and $rn$ is even (so that the set $\cG(n,r)$ of $r$-regular $n$-vertex graphs is non-empty), let $f_r^*(n)  := \min_{G\in \cG(n,r)} \min_{\emptyset \neq A \subseteq V(G)} f_r(A, G)$.

\begin{lemma} \label{lem.fstar}
For $n \geq r+1$ with $rn$ even; $f_r^*(n) = \frac{r+1}{n}$, and $f_r(A,G)$ achieves this minimum exactly when $G[A]=K_{r+1}$ (and this is a component of $G$).
\end{lemma}
\begin{proof}
Firstly note that if $G[A]=K_{r+1}$ then clearly $f_r(A,G)=(r+1)/n$. Now fix an $r$-regular graph $G$ on $n$ vertices, and let $A$ be a subset of the vertices such that $G[A]\neq K_{r+1}$. To prove the claim it will suffice to show that
\begin{equation}\label{eq.exactlywhat} \frac{e(A, \bar{A}) n }{r|A|} + |A| > r+1. 
\end{equation}

If $|A|>r+1$ then it is easy to see that~\eqref{eq.exactlywhat} holds. Similarly if $|A|=r+1$ then there must be edges from $A$ to the rest of the graph as we assumed $G[A]\neq K_{r+1}$, but $e(A,\bar{A})>0$ and $|A|=r+1$ together imply that~\eqref{eq.exactlywhat} holds, so we can assume that $|A|<r+1$. Set $|A|=r+1-\ell$ for some $\ell>0$. Observe that because $G$ is $r$-regular we have $e(A, \bar{A})\geq (r+1-\ell)\ell$. This is because each of the $|A|=r+1-\ell$ vertices can have at most $r-\ell$ of their incident edges within the part $A$ and so at least $\ell$ must be external. Hence,
$$\frac{e(A, \bar{A}) n }{r|A|} + |A| \geq \frac{(r+1-\ell)\ell n}{r(r+1-\ell)} +r+1-\ell = r+1 + \ell\Big(\frac{n}{r}-1\Big) \, > r+1$$
since $n>r$. This completes the proof of the lemma.
\end{proof}

\begin{proof}[Proof of Proposition~\ref{prop.maxqr} part (a)]
\hspace{.1in}

The sum in the expression for modularity in~\eqref{eq.modis} is a weighted average of terms $f_r(A,G)$ and so
$$q^+_r (n) = 1- \min_{G \in \cG(n,r) } \min_{\cA} \sum_{A \in \cA} \frac{|A|}{n} f_r(A,G) \leq 1-f_r^*(n)$$
with equality iff there is an $r$-regular graph $G$ with a vertex partition $\cA$ such that $\forall A\in \cA$, $f_r(A,G)=f_r^*(n)$. By the claim this means that a graph achieves the bound iff there is a vertex partition into disjoint copies of $K_{r+1}$. Hence if $(r+1)|n$, then $q^+_r(n) = 1-\frac{r+1}{n}$ with unique optimum graph the disjoint union of copies of $K_{r+1}$. Finally, for the case when $(r+1)$ does not divide $n$, for any graph on $n$ vertices there must be some part $A$ with $|A|\neq r+1$, and by the claim $f_r(A,G)>(r+1)/n$; and because modularity is a weighted average of such terms over all parts, we have $q^+_r(n)<1-(r+1)/n$.
\end{proof}

\begin{proof}[Proof of Proposition~\ref{prop.maxqr} part (b)] \hspace{.1in}

First we show that $g_r(n)$ (defined in~\eqref{eqn.qplus}) is a lower bound on $q_r^+(n)$ for certain values of $n$.  We see that if $r$ is even and $n \geq r(r+1)$, or if $r$ is odd and $n \geq (r-1)(r+3)/2$ (and $rn$ is even), then we have $q^+_r(n) \geq g_r(n)$. 

Consider the subcase when $r$ is even. Recall that we write $n$ as $a(r+1)+b$ where $a$ is a non-negative integer and  $0 \leq b \leq r$. Assume that $a \geq b$, which must hold when $n \geq r(r+1)$. Let $H_{r+2}$ be $K_{r+2}$ less a perfect matching, which is an $r$-regular graph on $r+2$ vertices.  Think of $n$ as $b(r+2)+(a-b)(r+1)$. Let $G^*_n$ be the $n$-vertex graph consisting of $b$ copies of $H_{r+2}$ and $a-b$ copies of $K_{r+1}$, all disjoint. Then with connected components partitions $\cC$,
\begin{eqnarray*}
q_{\cC}^D(G^*_n) & = &
\frac1{n^2} \left( b(r+2)^2 + (a-b) (r+1)^2 \right)\\
& = &
\frac1{n^2} \left( b(r+2)(r+1) + b(r+2)  + (a-b) (r+1)^2 \right)\\
& = &
\frac1{n^2} ( (r+1)n + b(r+2) )
\; = \; 1 - g_r(n).
\end{eqnarray*}
Hence $q^+_r(n) \geq q_{\cC}(G^*_n) = g_r(n)$, as required.

Now consider the subcase when $r \geq 3$ is odd. Observe $b$ must be even (and so $b \leq r-1$).  Assume that $a \geq b/2$, which must hold when $n \geq (r-1)(r+1)/2$.  
Let $H_{r+3}$ be $K_{r+3}$ less a 2-factor (that is, a spanning subgraph with each vertex degree 2, for example a Hamilton cycle), which is an $r$-regular graph on $r+3$ vertices.  (There is no $r$-regular graph with $r+2$ nodes, since $r+2$ is odd.)
Think of $n$ as $(r+3)b/2 + (r+1)(a-b/2)$.  Let $H_n^*$ be the $n$-vertex graph  consisting of $b/2$ copies of $H_{r+3}$ and $a-b/2$ copies of $K_{r+1}$, all disjoint. 
Then
\begin{eqnarray*}
q_{\cC}^D(H^*_n) & = &
\frac1{n^2} \left( (b/2)(r+3)^2 + (a\!- b/2) (r+1)^2 \right)\\
& = &
\frac1{n^2} \left( (b/2)(r+3)(r+1) + b(r+3)  + (a\!- b/2) (r+1)^2 \right)\\
& = &
\frac1{n^2} ( (r+1)n + b(r+3) )
\; = \; 1 - g_r(n).
\end{eqnarray*}
Hence $q^+_r(n) \geq q_{\cC}(H^*_n) = g_r(n)$, as required.

We have now shown in both cases that $q^+_r(n) \geq g_r(n)$ when claimed. 
Observe that
if $r$ is even, then $\frac{b(r+2)}{n^2} \leq \frac{r(r+2)}{n^2}$; and if $r$ is odd, then $\frac{b(r+3)}{n^2} \leq \frac{(r-1)(r+3)}{n^2} \leq \frac{r(r+2)}{n^2}$.  Hence we always have
\begin{equation} \label{eqn.grngeq} 
g_r(n) \geq 1- \frac{r+1}{n} - \frac{r(r+2)}{n^2}.
\end{equation}
The inequality~(\ref{eq.approxbestmod}) follows from inequality~(\ref{eqn.grngeq}), since $q^+_r(n) \geq g_r(n)$ for $n \geq r(r+1)$, as we have seen.
\smallskip

Next we use the above result to show that we need only to consider regular graphs with `small' connected components, and thus we need only to consider the connected components partition $\cC$.

Let $G \in \cG(n,r)$, and let $\cA$ be any vertex partition with a part $A'$ of size $a =|A'| \geq 2(r+1)$. Observe that $f_r(A',G) \geq a/n$, and for the other parts $A \in \cA$, $f_r(A,G) \geq (r+1)/n$ by the claim in the previous proof. Thus
\begin{eqnarray*}
1- q_{\cA}(G) &=&
\sum_{A \in \cA} \frac{|A|}{n} f_r(A,G)
\; \geq \; \frac{n-a}{n} \left(\frac{r+1}{n} \right) + \frac{a}{n} \cdot \frac{a}{n}\\
&=&
\frac{r+1}{n} + \frac{a}{n} \, \frac{a-(r+1)}{n}
\; \geq \;
\frac{r+1}{n} + \frac{2(r+1)^2}{n^2}. 
\end{eqnarray*}
Hence, by~(\ref{eqn.grngeq}) and the first part of the proposition,
\[ q_{\cA}(G) \leq 1- \frac{r+1}{n} - \frac{2(r+1)^2}{n^2} < g_r(n) \leq q_r^+(n)\]
if $n \geq r(r+1)$.
Therefore, for such $n$, if $q_{\cA}(G) = q^+_r(n)$, then each part of $\cA$ must have size at most $2r+1 < 2(r+1)$.

Now let $G$ have a connected component $H$ of size $s \geq 6 (r+1)$.  Suppose that $\cA$ is a partition for $G$ with $q_{\cA}(G) = q_r^+(n)$.  Then, by the above, $\cA$ must break $G$ into at least $s/(2(r+1))$ parts, and so there are at least $s/(2(r+1)) -1$ cross-edges for $\cA$.  Hence
\[ 1- q_{\cA}(G) \geq  \left(\frac{s}{2(r+1)}-1\right) \frac2{rn} + \frac{n-s}{n} \frac{r+1}{n} = \frac{r+1}{n} + \frac{s}{n} \left(\frac1{r(r+1)} - \frac{r+1}{n} \right) - \frac2{rn}.  \]
Now suppose that $n \geq 2r(r+1)^2$.
Then
\[ 1- q_{\cA}(G) \geq \frac{r+1}{n} + \frac{s}{n} \left(\frac1{2r(r+1)} \right) - \frac2{rn} 
\geq \frac{r+1}{n} + \frac1{rn}.\]
But $\frac1{rn} > \frac{r(r+2)}{n^2}$ (since $n > r^2(r+2)$), so by~(\ref{eqn.grngeq}) we have
$1- q_{\cA}(G) > 1-g_r(n)$, and so $
q_{\cA}(G) < g_r(n) \leq q^+_r(n)$.

Thus we need only to consider regular graphs with all components of size less than $6(r+1)$.
For an $r$-regular graph with a connected component $H$ on $h$ vertices, if $h< 2 \sqrt{n/r}$  (the `resolution limit' \cite{FortBart2008}) then no optimal partition breaks up this component. For, if some optimal partition breaks $H$ into $i \geq 2$ parts, then
\begin{eqnarray*}
0 & \leq & 
\frac{(hr)^2-i(hr/i)^2}{(nr)^2} - \frac{i-1}{nr/2} 
\; = \; \frac{i-1}{n} \left(\frac{h^2}{i n} - \frac2{r} \right),
\end{eqnarray*}
so $h^2 \geq 2i n/r \geq 4n/r$.  Thus for graphs with maximum component size $h$ and $n > rh^2/4$ vertices, the partition into connected components is the unique optimal partition.

Hence, for 
$n \geq 9 r (r+1)^2$, we need only to consider graphs with the connected components partition~$\cC$.  But now by the strict convexity of $f(x)=x^2$, we see that each graph with the maximum modularity must have at most two sizes of components, differing by one if $r$ is even and by 2 if $r$ is odd; and then it is easy to see that these must be the smallest two such sizes.
Hence the graphs $G^*_n$ and $H^*_n$ constructed earlier achieve the maximum modularity, and any graph achieving the maximum modularity must be of this form.
The only flexibility is in the choice of the missing 2-factor in the graphs $H_{r+3}$. This completes the proof of part (b), and thus of the whole of Proposition~\ref{prop.maxqr}.
\end{proof}

In Proposition~\ref{prop.maxqr} part (b), it is not hard to improve the bound $n \geq 9 r (r+1)^2$ by a constant factor, though possibly our bound is very pessimistic and $n \geq r(r+1)$ may suffice.


\section{Concluding remarks}
\label{sec.concl}

Two main contributions of this paper are (a) the numerical bounds for the typical modularity of random regular graphs $\Gr$ for $3 \leq r \leq 12$, which can be used to investigate the  significance of observed clustering, together with determining the asymptotic behaviour for large $r$; and (b) showing high modularity for graphs in families such as the low degree treelike graphs, and thus for example for random planar graphs. We also investigated the minimum modularity $q_r^-(n)$, and the maximum modularity $q_r^+(n)$.


Proposition~\ref{cor.allrreg} shows that the modularity of any large cubic graph is at least $2/3 -o(1)$; and in our simulations the highest modularity value of a partition found for a random cubic graph was about $0.68$ (though there is no guarantee this is close to optimal).
Does the minimum modularity $q_3^{-}(n) \to \tfrac23$ as $n \to \infty$? Perhaps we even have $\q(G_{n,3}) \to \frac23$ in probability as $n \to \infty$? If so, then random cubic graphs would give extremal examples for low modularity, but is it the case?  The opposite conjecture was presented by the first author at the Bellairs workshop on Probability, Combinatorics and Geometry in April 2016.
 
\begin{conj}\label{conj.cubic} 
There exists $\delta>0$ such that $q^*(G_{n,3}) \geq 2/3 + \delta$ whp.
\end{conj}

Our results on cubic graphs highlight the importance of choosing the right baseline for modularity to assess significance.
For suppose we have a cubic network $G$ with $n$ vertices, where $n$ is large.  We have seen in Corollary~\ref{cor.allrreg} and Theorem~\ref{thm.kwus} that $q^*(G) > 0.66$, just because $G$ is cubic.  Hence, unless $q_0>0.66$, the fact that a partition ${\cal A}$ has $q_{\cal A}(G) =q_0$ should certainly not be considered as evidence that $\cal A$ shows community structure. 
\smallskip

We have discussed the modularity of random regular graphs, $\q(\Gr)$. This is very different from considering the modularity of a random partition of a fixed graph.  For let $G$ be any fixed graph (with at least one edge).  Fix $k \geq 2$, and suppose that we generate a random partition $\cal A$ of the vertices by placing each vertex into one of $k$ parts independently with probability $1/k$.  Then
\begin{equation} \label{eqn.randpart}
{\mathbb E}[q_{\cal A}(G)] <  0.
\end{equation}
Thus comparing the modularity of a partition which we find to that of a random partition is likely to give a false positive! To see why~(\ref{eqn.randpart}) holds, observe that for an edge $e$ in $G$ the probability that both endpoints are placed in the same part is $1/k$, so ${\mathbb E}[q_{\cal A}^E(G)]=1/k$; and by~(\ref{eqn.degtaxreg}) the degree tax is always at least $1/k$, and sometimes larger. 

\smallskip
In a companion paper~\cite{ERus} (see also~\cite{thesis}) we prove that there is a phase transition for the modularity of Erd\H{o}s-R\'enyi random graphs. Additionally we show in~\cite{latticeus} that large subgraphs of lattices, and more generally all large graphs which embed in space with `small distortion', have high modularity.

\smallskip

\noindent
{\bf Acknowledgement}  Thanks to Michael Krivelevich for comments at the meeting Combinatorics Downunder in Melbourne in 2016, which led us to Lemma~\ref{lem.ub1}.

\bibliographystyle{plain}

\begin{thebibliography}{10}

\bibitem{alon1997edge}
Alon, N.\ (1997) On the edge-expansion of graphs. 
{\em Combinatorics, Probability and Computing}, 
\textbf{6}(2), 145--152.

\bibitem{alon1988explicit}
Alon, N.\ \& Chung, F.\ R.\ K.\ (1988)
Explicit construction of linear sized tolerant networks.
{\em Annals of Discrete Mathematics}, \textbf{38}, 15--19.

\bibitem{alon2000probabilistic}
Alon, N.\ \& Spencer, J.\ H.\ (2000)
The probabilistic method.
{\em Wiley--Intersci. Ser. Discrete Math. Optim}.

\bibitem{bagrow}
Bagrow, J.\ P.\ (2012)
Communities and bottlenecks: Trees and treelike networks have high modularity.
{\em Physical Review E}, \textbf{85}(6):066118 

\bibitem{louvain}
Blondel, V.\ D., Guillaume, J.\ L., Lambiotte, R.\ \& Lefebvre, E.\ (2008)
Fast unfolding of communities in large networks.
{\em Journal of Statistical Mechanics: Theory and Experiment},
\textbf{10} P10008.

\bibitem{karb}
Bodlaender, H.\ L.\ (1998)
A partial $k$-arboretum of graphs with bounded treewidth.
{\em Theoretical Computer Science}, \textbf{209}(1) 1--45.

\bibitem{bolla2013largest}
Bolla, M., Bullins, B., Chaturapruek, S., Chen, S.\ \& Friedl, K.\ (2013)
When the largest eigenvalue of the modularity and normalized
  modularity matrix is zero.
{\em preprint arXiv:1305.2147}.

\bibitem{nphard}
Brandes, U., Delling, D., Gaertler, M., Gorke, R., Hoefer, M.,
 Nikoloski, Z.\ \& Wagner, D.\ (2008)
On modularity clustering.
{\em Knowledge and Data Engineering, IEEE Transactions on},
\textbf{20}(2) 172--188.

\bibitem{dasgupta2013complexity}
Das{G}upta, B.\ \& Desai, D.\ (2013)
On the complexity of {N}ewman's community finding approach for biological and social networks.
{\em Journal of Computer and System Sciences}, \textbf{79}(1) 50--67.

\bibitem{modgraphclasses}
De~Montgolfier, F., Soto, M.\ \& Viennot, L. (2011)
Asymptotic modularity of some graph classes.
In {\em Algorithms and Computation}, 435--444.

\bibitem{dembo}
Dembo, A., Montanari, A. \& Sen, S.\ (2017)
Extremal cuts of sparse random graphs.
{\em The Annals of Probability} \textbf{45}(2) 1190-1217.

\bibitem{diaz2007bounds}
D{\'\i}az, J., Serna, M. J.\ \& Wormald, N.\ (2007)
Bounds on the bisection width for random {$d$}-regular graphs.
{\em Theoretical Computer Science}, \textbf{382}(2) 120--130.

\bibitem{norin}
Dvorak, Z.\ \& Norin, S.\ (2014)
Treewidth of graphs with balanced separations.
\newblock {\em preprint arXiv:1408.3869}.

\bibitem{fortunato2010community}
Fortunato, S.\ (2010)
Community detection in graphs.
{\em Physics Reports}, \textbf{486}(3) 75--174.

\bibitem{FortBart2008}
Fortunato, S.\ \& Barth{\'e}my, M. (2007)
Resolution limit in community detection.
{\em Proceedings of the National Academy of Sciences}, \textbf{104}(1) 36--41.

\bibitem{fortunato2016community}
Fortunato, S.\ \& Darko Hric, D.\ (2016)
Community detection in networks: A user guide.
{\em Physics Reports}, \textbf{659} 1--44.

\bibitem{frankearxiv}
Franke, B.\ \& Wolfe, P.~J.\ (2016) 
Network modularity in the presence of covariates.
{\em preprint arXiv:1603.01214}.

\bibitem{friedman2008proof}
Friedman, J. (2008)
\newblock {\em A proof of Alon's second eigenvalue conjecture and related
  problems}, volume 195, No.\ 910, p.\ viii+100 of {\em Memoirs of the American
  Mathematical Society}.
\newblock American Mathematical Society.

\bibitem{gilbert1984separator}
Gilbert, J.\ R., Hutchinson, J. P.\ \& Tarjan, R.\ E.\ (1984)
A separator theorem for graphs of bounded genus.
{\em Journal of Algorithms}, \textbf{5}(3) 391--407.

\bibitem{GPA04}
Guimer{\`a}, R., Sales-Pardo, M.\ \& Amaral, L.~A.~N. (2004)
Modularity from fluctuations in random graphs and complex networks.
{\em Physical Review E}, \textbf{70}:025101.

\bibitem{jbook}
Janson, S., {\L}uczak, T.~\& Ruci{\'n}ski, A.~(2011)
{\em Random {G}raphs}, vol.~45,
John Wiley \& Sons.

\bibitem{jeubcode}
Jutla, I.~S., Jeub, L.~G.~S.~\& Mucha, P.~J. (2011)
A generalized louvain method for community detection implemented in
  matlab.
{\em URL http://netwiki.amath.unc.edu/GenLouvain}.

\bibitem{kleinberg2005textbook}
Kleinberg, J.~\& Tardos, E.~ (2005)
{\em Algorithm {D}esign}.
{A}ddison {W}esley.

\bibitem{KolWorm}
Kolesnikm B.~\& Wormald N.~ (2014)
Lower bounds for the isoperimetric numbers of random regular graphs.
{\em SIAM Journal on Discrete Mathematics}, \textbf{28}(1) 553--575. 

\bibitem{popular}
Lancichinetti, A.~\& Fortunato, S. (2011)
Limits of modularity maximization in community detection.
{\em Physical Review E}, \textbf{84}(6):066122.

\bibitem{martelot}
Le~{M}artelot, E.~\& Hankin, C.~ (2013)
Fast multi-scale detection of relevant communities in large-scale
  networks.
{\em The {C}omputer {J}ournal}, 56(9):1136--1150.

\bibitem{majstorovic2014note}
Majstorovic, S.~\& Stevanovic, D. (2014)
A note on graphs whose largest eigenvalues of the modularity matrix
  equals zero.
{\em Electronic Journal of Linear Algebra}, \textbf{27}(1) 256.

\bibitem{maxDegSurface}
Mc{D}iarmid, C.~\& Reed, B., (2008)
On the maximum degree of a random planar graph.
{\em Combinatorics, Probability and Computing}, \textbf{17}(4) 591--601.

\bibitem{McDSker}
Mc{D}iarmid, C.~\& Skerman, F.~ (2013)
Modularity in random regular graphs and lattices.
{\em Electronic Notes in Discrete Mathematics}, \textbf{43} 431--437.

\bibitem{ERus}
Mc{D}iarmid, C.~\& Skerman, F.~ (2017+)
Modularity of {Erd{\H{o}}s}-{R}\'enyi graphs.
{\em preprint}.

\bibitem{latticeus}
Mc{D}iarmid, C.~\& Skerman, F.~ (2017+)
Modularity of lattices and other well-embeddable graphs.
{\em preprint}.

\bibitem{NewmanGirvan}
Newman, M.~E.~J.~\& Girvan, M.~(2004)
Finding and evaluating community structure in networks.
{\em Physical Review E}, \textbf{69}(2) 026113.

\bibitem{porter2009communities}
Porter, M.~A., Onnela, J.~K.~\& Mucha, P.~J.~(2009)
Communities in networks.
{\em Notices of the AMS}, \textbf{56}(9) 1082--1097.

\bibitem{pralat}
Prokhorenkova, L.~O., Pra{\l}at, P.~\& Raigorodskii, A.~(2017)
Modularity of models of complex networks.
{\em preprint arXiv:1701.03141}.

\bibitem{reichardt2006statistical}
Reichardt, J.~\& Bornholdt, S.~(2006)
Statistical mechanics of community detection.
{\em Physical Review E}, \textbf{74}(1) 016110.

\bibitem{trulymodular}
Reichardt, J.~\& Bornholdt, S.~(2006)
When are networks truly modular?
{\em Physica D: Nonlinear Phenomena}, \textbf{224}(1) 20--26.

\bibitem{thesis}
Skerman, F.~(2016)
{\em Modularity of Networks}.
PhD thesis.

\bibitem{randQuick}
Steger, A.~\& Wormald, N.~C.~(1999)
Generating random regular graphs quickly.
{\em Combinatorics, Probability and Computing}, \textbf{8} 377--396.

\bibitem{trajanovski2012maximum}
Trajanovski, S., Wang, H.~\& Van~Mieghem, P.~(2012)
Maximum modular graphs.
{\em The European Physical Journal B-Condensed Matter and Complex
  Systems}, \textbf{85}(7) 1--14.

\bibitem{van2010spectral}
Van~Mieghem, P., Ge, X., Schumm, P., Trajanovski, S.~\& Wang, H.~(2010)
Spectral graph analysis of modularity and assortativity.
{\em Physical Review E}, \textbf{82}(5) 056113.

\end{thebibliography}

\end{document}